\RequirePackage[l2tabu, orthodox]{nag} 
\documentclass[11pt, english, toc=bib]{scrartcl}
\addtokomafont{disposition}{\rmfamily}
\addtokomafont{descriptionlabel}{\rmfamily}
\usepackage[T1]{fontenc}
\usepackage[utf8]{inputenc}
\usepackage[english]{babel}
\usepackage[height=25cm, width=16cm, bindingoffset=1cm]{geometry} 

\usepackage{latexsym} 			
\usepackage{amsfonts} 			
\usepackage{amssymb} 			
\usepackage{amsmath}
\usepackage{stmaryrd}			
\usepackage{mathtools} 			
\usepackage{physics}			
\usepackage[bbsets, subsetnonamb]{jkmath}
\usepackage{yhmath}				
\usepackage[greeklower, frac, autobold, root, rootclose]{mathfixs}
\usepackage{mathdots} 			
\usepackage{lmodern}

\usepackage{bm}
\usepackage{blindtext} 			
\usepackage[autostyle]{csquotes}
\usepackage{todonotes}			
\usepackage[final, unicode, hypertexnames=false]{hyperref} 	
\usepackage{footnotebackref} 	
\usepackage{microtype}			
\usepackage{paralist}			
\usepackage{pifont}				
\usepackage{stackrel}			
\usepackage[fixlanguage]{babelbib}			
\usepackage{soul}				
\usepackage[normalem]{ulem}
\usepackage{pgf,tikz,pgfplots}
\usepackage{tikz-cd}			
\pgfplotsset{compat=1.15}
\usetikzlibrary{arrows, babel, calc}
\usepackage{mdframed}
\usepackage[thmmarks, amsmath, hyperref, framed]{ntheorem}			
\usepackage{xcolor}
\usepackage{cancel}				
\usepackage{tensor}				
\usepackage{siunitx}			
\usepackage{pdflscape}			
\usepackage{array} 				
\usepackage{longtable}			
\usepackage{listings,listingsutf8} 
\usepackage[os=win]{menukeys}
\usepackage{tikz-cd}			
\usepackage[figurewithin=none]{caption}			
\usepackage{subcaption}			
\usepackage[section]{placeins}	
\usepackage{xspace}				
\usepackage{gradient-text}
\usepackage{pdfpages}
\usepackage{tabularx}
\usepackage{booktabs}
\usepackage[english]{cleveref}	
\usepackage{autonum}			

\definecolor{comment}{HTML}{005000} 
\definecolor{string}{HTML}{BB00BB}
\deffootnote{1.5em}{1em}{ 
	\textsuperscript{\hyperref[\BackrefFootnoteTag]{\thefootnotemark}}\,%
}

\newtheoremstyle{Stil}{\item[ \normalfont\textbf{##1 ##2. }]}{\item[\normalfont\textbf{##1 ##2 (##3).}]}
\theorembodyfont{\itshape}

\theoremstyle{Stil}

\newtheorem{Satz}{Theorem}[section]
\newtheorem{Def}[Satz]{Definition}
\newtheorem{Bem}[Satz]{Remark}
\newtheorem{Thm}[Satz]{Theorem}
\newtheorem{Lem}[Satz]{Lemma}

\newtheorem{Kor}[Satz]{Corollary}

\newtheorem{Ann}[Satz]{Assumption}

\AddToHook{env/Def/begin}{\crefalias{Satz}{Def}}
\AddToHook{env/Bem/begin}{\crefalias{Satz}{Bem}}
\AddToHook{env/Thm/begin}{\crefalias{Satz}{Thm}}
\AddToHook{env/Lem/begin}{\crefalias{Satz}{Lem}}
\AddToHook{env/Konst/begin}{\crefalias{Satz}{Konst}}
\AddToHook{env/Bsp/begin}{\crefalias{Satz}{Bsp}}
\AddToHook{env/Kor/begin}{\crefalias{Satz}{Kor}}
\AddToHook{env/Prop/begin}{\crefalias{Satz}{Prop}}
\AddToHook{env/Ann/begin}{\crefalias{Satz}{Ann}}

\newtheoremstyle{StilNoNumber}{\item[ \normalfont\textbf{##1. }]}{\item[\normalfont\textbf{##1 (##3).}]}
\theorembodyfont{\itshape}

\theoremstyle{StilNoNumber}

\newtheorem{ThmNN}{Theorem}[section]

\newtheoremstyle{StilB}{\item[ \textit{##1: }]}{\item[ \textit{##3:}]}
\theorembodyfont{\normalfont}
\theoremstyle{StilB}

\theoremsymbol{\ensuremath{\square}}
\newtheorem{proof}{Proof}

\makeatletter
\autonum@generatePatchedReferenceCSL{Cref} 

\newcommand*{\centerfloat}{
	\parindent \z@
	\leftskip \z@ \@plus 1fil \@minus \textwidth
	\rightskip\leftskip
	\parfillskip \z@skip}

\AtBeginDocument{
	\expandafter\renewcommand\expandafter\subsection\expandafter{%
		\expandafter\@fb@secFB\subsection
	}%
}

\makeatother

\newlength{\leftstackrelawd} 
\newlength{\leftstackrelbwd}
\def\leftstackrel#1#2{\settowidth{\leftstackrelawd}%
	{${{}^{#1}}$}\settowidth{\leftstackrelbwd}{$#2$}%
	\addtolength{\leftstackrelawd}{-\leftstackrelbwd}%
	\leavevmode\ifthenelse{\lengthtest{\leftstackrelawd>0pt}}%
	{\kern-.5\leftstackrelawd}{}\mathrel{\mathop{#2}\limits^{#1}}}

\newcommand{\inv}{^{-1}}

\newcommand{\hilbert}{\ensuremath{\mathcal{H}}}
\newcommand{\vilbert}{\ensuremath{\mathcal{V}}}
\newcommand{\laplace}{\Delta}

\renewcommand{\pmat}[1]{\begin{pmatrix} #1 \end{pmatrix}}

\newcommand*{\tran}{^{\mkern-1.5mu\mathsf{T}}}

\newcommand{\leqc}{\lesssim}

\let\temp\phi
\let\phi\varphi
\let\varphi\temp
\let\temp\epsilon
\let\epsilon\varepsilon
\let\varepsilon\temp
\let\temp\rho
\let\rho\varrho
\let\varrho\temp
\let\temp\theta
\let\theta\vartheta
\let\vartheta\temp

\let\temp\Re
\let\Re\real
\let\real\temp
\let\temp\Im
\let\Im\imaginary
\let\imaginary\temp

\DeclarePairedDelimiterX{\skpu}[1]{\langle}{\rangle}{#1}
\newcommand{\skp}[1]{\skpu*{#1}} 

\let\div\relax
\DeclareMathOperator{\div}{div}
\DeclareMathOperator{\chak}{char}

\newcommand{\amax}{{a_{\text{max}}}}

\begin{document}
	\title{A Wave-type Model for Age- and Space-structured Epidemics}
	\author{Nicolas Schlosser}
	\date{\today}
	\maketitle
	
	\begin{abstract}
		We introduce a novel approach of epidemic modeling by combining age-structured models with damped wave equations. This transforms the parabolic-type reaction-diffusion model into a hyperbolic system that shares many properties with a wave or telegrapher's equation. After we establish the existence of a weak solution of the resulting partial differential equation by means of characteristics, we show that the solutions to the new model converge to a solution of the standard age-dependent reaction-diffusion equation when we let the wave parameter become arbitrarily small. We conclude with a numerical example to illustrate the behavior of the new model and to further support our findings.
	\end{abstract}



\section{Introduction}

Since the beginning of mathematical epidemiology, differential equations have been a valuable tool to predict the course of a disease, and already the first epidemic model \begin{equation}
	\dot{S} = - \lambda I S, \quad \dot{I} = \lambda I S - \gamma I, \quad \dot{R} = \gamma I,
\end{equation} 
introduced in \cite{KermackMcKendrick}, where $S$, $I$, and $R$ stand for susceptible, infectious and removed individuals respectively, has the form of a nonlinear differential equation. Over time, more elaborate models were proposed for a large variety of diseases, but the general structure of the models remained the same. In order to obtain more accurate models and to account for phenomena such as vertical transmission, diseases spanning multiple generations, or the spread of an epidemic over a country, more variables have been added to the models, resulting in partial differential equations instead. Among the most frequently added variables are age and location of the affected individuals, and the main topic of this work are models with time, age and space. The general structure of an epidemic model with these variables can be stated as \begin{subequations} \label{eq:GeneralModel}
	\begin{gather+}
		(\partial_t + \partial_a) y + L(a, x) y + \Lambda(a, x, y) y = \sigma(a) \laplace y, \label{eq:IntroState}\\
		y(t = 0) = y_0, \quad \partial_\nu y(x \in \partial \Omega) = 0, \label{eq:IntroInit}\\
		y(a = 0) = \int_0^\amax \beta(\alpha, x) y(t, \alpha, x) \dd{\alpha}. \label{eq:IntroBirth}
	\end{gather+}

where $t$ is time, $a \in [0, \amax]$ represents age, and $x \in \Omega$ is the space variable. Further, $y \in \R^n$ is a vector consisting of the $n$ compartments of the model (for example $y = (S, I, R)$), $L$ gathers the linear terms and $\Lambda(y)$ the nonlinear nonlocal terms that, for example, capture infection processes. The matrix $\Lambda(y) \in \R^{n \times n}$ is defined in such a way that every entry $\Lambda(y)^{hi}$ ($h$, $i = 1, \ldots, n$) can be obtained by integrating $y$ over a kernel $k^{hi}$, that is to say \begin{equation} \label{eq:LambdaDefinition}
	\Lambda(a, x, y)^{hi} = \sum_{j = 1}^{n} \int_0^\amax \int_\Omega k^{hij}(a, x, \alpha, \xi) y_j(\alpha, \xi) \dd{\xi} \dd{\alpha} = \skp{k^{hi}(a, x, \cdot, \cdot), y}_{L^2((0, \amax) \times \Omega)^n}.
\end{equation}
\end{subequations}
Here, the function $k^{hij}(a, x, \alpha, \xi)$ describes how many susceptible individuals from compartment $y_i$ having age $a$ and being located at position $x$ are infected from infectives $y_j$ with age $\alpha$ and location $\xi$ and subsequently transition into the infected class $y_h$. Examples for this model class are well known. Going back to \cite{SchlosserMaster}, in \cite{AzmiSchlosser} we gave a general existence and uniqueness result.

However, a common property of diffusion models is the so-called infinite propagation speed, which in the case of epidemic models states that even if the infectious individuals are initially concentrated in a subdomain of $\Omega$, after any time period, however short it may be, the number of infective individuals is positive in every point of $\Omega$. While this effect might be small when computing solutions to epidemic models with real-world data, it nonetheless motivates and justifies the search for alternative formulations. A common way to obtain models with finite propagation speed is to introduce a so-called relaxation parameter $\tau$, similar to how a telegrapher's equation arises from a diffusion equation (cf. \cite{Hadeler2}). Together with some suitable initial and boundary conditions this yields the model \begin{subequations} \label{eq:relaxedEquation}
	\begin{gather+}
		(1 + \tau (\partial_t + \partial_a)) ((\partial_t + \partial_a) y + L y + \Lambda(y)) y = \sigma(a) \laplace y,\label{eq:RelState}\\
		y(t = 0) = y_0, \quad (\partial_t + \partial_a) y(t = 0) = y_1, \quad \partial_\nu y(x \in \partial \Omega) = 0,\label{eq:RelInit}\\
		y(t, a = 0) = \int_0^\amax \beta_0(\alpha, x) y(t, \alpha, x) \dd{\alpha} + g_0(t, x),\label{eq:RelBirth1}\\
		\begin{aligned}
				((\partial_t + \partial_a) y)(a = 0) &= \int_0^\amax \beta_L(\alpha, x) y(t, \alpha, x) + \beta_\nabla(\alpha, x) \cdot \nabla y(t, \alpha, x) + \beta_1(\alpha, x) (\delta y)(t, \alpha, x)\\
				&\quad + \qty(\beta_1(\alpha, x) \Lambda(\alpha, x, y) - \Lambda(0, x, y) \beta(\alpha, x)) y(t, \alpha, x) \dd{\alpha}\\
				&\quad - \Lambda(0, y) g_0(t, x) + g_1(t, x)
		\end{aligned} \label{eq:RelBirth2}
	\end{gather+}
\end{subequations}
which will be derived in detail in \Cref{sec:RelMotivation}. Here, $\beta_0$ is a birth rate, $g_0$ an explicitly given additional number. of births, and $\beta_1$ and $g_1$ are the analoga for the value $(\partial_t + \partial_a) y + L y + \Lambda(y) y$ on the boundary where $a = 0$. In what follows, we will prove that this system, or rather a reformulation of it, is well-posed and admits a weak solution for small time intervals:
\begin{ThmNN} \label{thm:Main1}
	Under certain regularity assumptions on the data, there exists a $T^* > 0$ such that for every $T \in (0, T^*)$ there is a unique weak solution $y \in C([0, T], \vilbert) \cap C(\bar{\mathcal{I}}, L^2((0, T), V))$ with $\delta y \in C([0, T], \hilbert) \cap C(\bar{\mathcal{I}}, L^2((0, T), H))$ to an appropriate weak formulation of the system \eqref{eq:relaxedEquation}. This $y$ satisfies the estimate \begin{equation}
		\norm{y}_{C([0, T], \vilbert)}^2 + \tau \norm{(\partial_t + \partial_a) y}_{C([0, T], \hilbert)}^2 \leqc \mathcal{K}
	\end{equation}
	where \begin{equation}
		\mathcal{K} \coloneqq \norm{y_0}_\vilbert^2 + \tau \norm{y_1}_\hilbert^2 + \norm{f}_{L^2((0, t), \hilbert)}^2 + \norm{g_0}_{L^2((0, t), V)}^2 + \tau \norm{g_1}_{L^2((0, T), H)}^2.
	\end{equation}
	More precisely, $T^*$ can be chosen as $C \cdot (\mathcal{K} + \tau \norm{g_0}_{L^\infty((0, T), H)}^2)\inv$ with a constant $C$ independent of $y_0$, $y_1$, $f$, $g_0$, $g_1$, or $\tau$.
\end{ThmNN}

Furthermore, we can relate solutions of the relaxed model \eqref{eq:relaxedEquation} to those of the unrelaxed \enquote{classical} model \eqref{eq:GeneralModel} if their initial and birth conditions match. More precisely, we can show that the relaxed model converges to the unrelaxed one for $\tau \to 0$, and the rate of convergence depends on the number of matching initial and birth conditions: \begin{ThmNN} \label{thm:Main2}
	Let $y^\tau$ be a weak solution of \cref{eq:relaxedEquation} and $y$ a weak solution of \cref{eq:GeneralModel}. Under certain regularity assumptions, suppose that $\beta_0 = q_1 \beta$ and $g_0 = (1 - q_1) y(a = 0)$ for some $q_1 \in \R$. Then there is a constant $C$ independent of $\tau$ (but possibly depending on $y$) such that \begin{equation}
		\tau \norm{(\partial_t + \partial_a) (y^\tau - y)}_{C([0, T], \hilbert)}^2 + \norm{y^\tau - y}_{C([0, T], \vilbert)}^2 \leq C \tau.
	\end{equation}
	Furthermore, if for some $q_2 \in \R$ the additional compatibility conditions \begin{gather}
		y_1 = \sigma \laplace y_0 - (L + \Lambda(y_0)) y_0, \quad \beta_1(\alpha) = q_2 \sigma(0) \beta(\alpha) \sigma(\alpha)\inv, \\ g_1 = (1-q_2) (\delta y + L y + \Lambda(y) y) (a = 0)
	\end{gather}
	hold, we even can estimate \begin{equation}
		\tau \norm{(\partial_t + \partial_a) (y^\tau - y)}_{C([0, T], \hilbert)}^2 + \norm{y^\tau - y}_{C([0, T], \vilbert)}^2 \leq C \tau^2
	\end{equation}
	with a constant $C$ independent of $\tau$, but possibly depending on $y$.
\end{ThmNN}

\subsection{Related works}

Models related to \eqref{eq:GeneralModel} are found abundantly in the literature. Comprehensive introductions to the topic are found in \cite{BDJ, LiYaMa, Perthame, Webb}. Age-structured models can be found in \cite{BDKW, KangRuan}, models with spatial structure in \cite{AbRaSch, Bongarti, WangZhangKuniya}, and the works \cite{Colombo, Fragnelli, Kuniya, Walker, Walker2} consider models where both are present. The method of relaxation was introduced in \cite{Maxwell} and \cite{Cattaneo}, and the limit behavior for relaxed models was e.g. studied in \cite{HuRacke}. In \cite{Kac}, wave equations were first introduced as a model for the movement of individuals, and much effort has since been expended into the generalization to higher dimensions, see for example \cite{Pogorui} and the references therein. Only a few works have addressed wave equations in population dynamics where the population moves along the real line, the most notable being \cite{AllharbiPetrovskii, BertagliaPareschi, Holmes, TillesPetrovskii}. Notably we mention \cite{Hadeler}, which speaks about epidemiological models, and \cite{Hadeler2}, which sketches the case of more than one space dimension. However, to the best of our knowledge, wave-type models with age structure, and the resulting birth conditions, have not been addressed yet. This work is a generalization of Chapter 4 of \cite{SchlosserThesis}.

\subsection{Structure of the paper}

This work is structured as follows: In \Cref{sec:RelMotivation}, we derive the model \eqref{eq:relaxedEquation} by applying a formal relaxation in time and age to \cref{eq:GeneralModel} and derive suitable boundary conditions for $a = 0$. \Cref{sec:Existence} is devoted to the precise formulation and proof of \Cref{thm:Main1}, and in \Cref{sec:Convergence} we give a precise formulation and proof of \Cref{thm:Main2}. Finally, \Cref{sec:Numerics} gives a numerical example for the main theorems and discusses the quantitative properties of the new model.

\subsection{Notation}

By $x^+ \coloneqq \max\set{x, 0}$, we denote the positive part of a real number $x$. For any open subset $\Omega \subset \R^n$ and $f: \Omega \to \R$, the integral of $f$ over $\Omega$ is denoted $\int_\Omega f(x) \dd{x}$. If $g: \partial \Omega \to \R$ is a function, where $\partial \Omega$ is the boundary of $\Omega$ which we assume to be sufficiently smooth, we denote the surface integral of $g$ over $\partial \Omega$ by $\oint_{\partial \Omega} g(x) \dd{A(x)}$.

For a Banach space $X$, we denote the norm of $X$ by $\norm{\cdot}_X$. For another Banach space $Y$, the notation $L(X, Y)$ denotes the bounded linear operators from $X$ to $Y$. Further we let $L(X) \coloneqq L(X, X)$ and $X' \coloneqq L(X, \R)$. All spaces are assumed to be real, until stated otherwise.

In the following we fix a smoothly bounded domain $\Omega \subset \R^d$ (in applications, typically $d$ equals two) and define the spaces 
\begin{equation}
	H \coloneqq L^2(\Omega)^n, \quad V \coloneqq H^1(\Omega)^n,
\end{equation}
where $n$ is the number of compartments in the epidemic model. By defining $V'$ as the dual of $V$ via the dual pairing induced by the inner product on $H$, it is easy to see that $V \subset H \subset V'$ is a Gelfand triple, in the sense of e.g. \cite[Def. 1.26]{HPUU}. By $\laplace \coloneqq \laplace_x$ we denote the Laplace operator acting on the space variables in $\Omega$. 

The age variable is assumed to be bounded by a finite maximal age, denoted by  $\amax > 0$. For brevity, we define $\mathcal{I} \coloneqq (0, \amax)$. We will frequently work in the spaces \begin{equation}
	\hilbert \coloneqq L^2(\mathcal{I}, H) = L^2(\mathcal{I} \times \Omega)^n, \quad \vilbert \coloneqq L^2(\mathcal{I}, V), \quad \vilbert' \coloneqq L^2(\mathcal{I}, V').
\end{equation}
To facilitate our analysis, we fix a time horizon $0 < T < \infty$. To keep the notation short, we write \begin{equation}
	\delta y \coloneqq (\partial_t + \partial_a) y,
\end{equation}
the expression being interpreted as a directional derivative. It is often useful to partition the variable space $[0, T] \times \overline{\mathcal{I}}$ into sets of the form
\begin{equation}
	\chak(t_0) \coloneqq \set{(t_0 + h, h) \where 0 \leq h \leq \amax} \cap ([0, T] \times \overline{\mathcal{I}})
\end{equation}
where $t_0 \in [-\amax, T]$. These sets represent the so-called \emph{characteristic lines} and are illustrated in \Cref{fig:CharakteristikTkA}. Often we restrict functions $\phi$ defined on $[0, T] \times \bar{\mathcal{I}}$ on these characteristics, for which we will use the notation $\restr{\phi}_{\chak(t_0)}(h) \coloneqq \phi(t_0 + h, h)$, for all parameters $h \in [\max\set{-t_0, 0}, \min\set{T-t_0, \amax}]$.

In estimates, we frequently use the notation $A \leqc B$ for an inequality of the form $A \leq c B$, where $c$ is a generic constant independent of the quantities to be estimated.

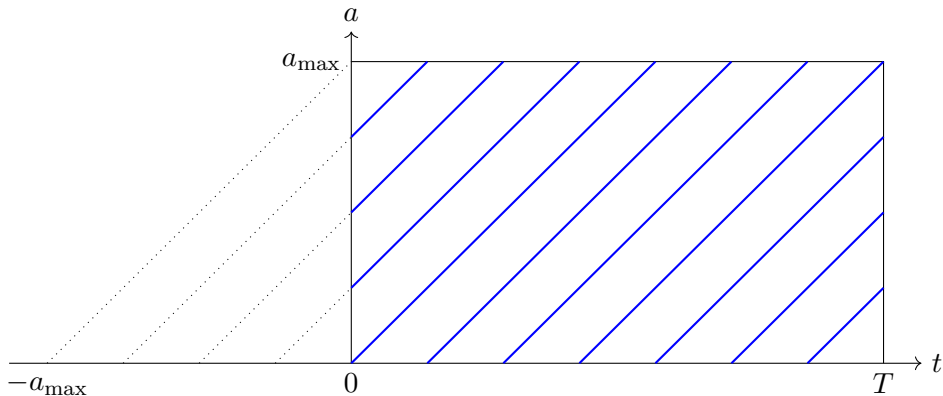
\begin{figure}
	\centering
	\begin{tikzpicture}
		\draw [->] (0,0) node[below]{0} -- (0,4.4) node[above]{$a$};
		\draw [->] (-4.5,0) -- (7.5,0) node[right]{$t$};
		\draw (0,4) node[left]{$\amax$} -- (7,4) -- (7,0) node[below]{$T$};
		\draw [dotted] (-4,0) node[below]{$-\amax$} -- (0,4);
		\draw [dotted] (-3,0) -- (0,3); \draw [blue, thick] (0,3) -- (1,4);
		\draw [dotted] (-2,0) -- (0,2); \draw [blue, thick] (0,2) -- (2,4);
		\draw [dotted] (-1,0) -- (0,1); \draw [blue, thick] (0,1) -- (3,4);
		\draw [blue, thick] (0,0) -- (4,4);
		\draw [blue, thick] (1,0) -- (5,4);
		\draw [blue, thick] (2,0) -- (6,4);
		\draw [blue, thick] (3,0) -- (7,4);
		\draw [blue, thick] (4,0) -- (7,3);
		\draw [blue, thick] (5,0) -- (7,2);
		\draw [blue, thick] (6,0) -- (7,1);
	\end{tikzpicture}
	\caption{Schematic image of the characteristics (blue lines). For every $t_0 \in [-\amax, T]$ we obtain a characteristic: If $t_0 \geq 0$, it starts in $(t, a) = (t_0, 0)$, for $t_0 < 0$ it starts in $(t, a) = (0, -t_0)$.}
	\label{fig:CharakteristikTkA}
\end{figure}

\section{Model derivation}\label{sec:RelMotivation}
To illustrate the relaxation process, we first consider the relation between the heat equation and the telegrapher's equation. Let $\Omega \subset \R^2$ be a domain with sufficiently smooth boundary $\partial \Omega$ and outer normal vector $\nu$, and let $u$ be some time- and space-dependent quantity. Then the change of $u$ is modeled by the continuity equation \begin{equation} \label{eq:ContEq}
	u_t + \div Q = f,
\end{equation}
where $Q$ is the flux of $u$, and $f$ is a source term (cf. \cite[eq. (1.1)]{LMPS}). If we assume that the flux is proportional to the negative gradient of the quantity, i.e. the quantity moves from higher to lower concentrations, and the bigger the difference the faster the movement, the flux takes the form \begin{equation} \label{eq:unrelFlux}
	Q = -\sigma \nabla u
\end{equation} where $\sigma > 0$ is some diffusion coefficient, and the model decouples to the diffusion equation \begin{equation}
	u_t - f = -\div (- \sigma \nabla u) = \sigma \laplace u.
\end{equation}
A suitable boundary condition is \begin{equation}\label{eq:FluxBoundary}
	\nu \cdot Q = 0,
\end{equation} i.e. there is no flux into or out of $\Omega$.  In absence of source terms $f$, this condition leads to conservation of quantity, since by the divergence theorem \begin{equation}
	\dv{t} \int_\Omega u(x) \dd{x} = -\int_\Omega \div Q(x) \dd{x} = -\oint_{\partial \Omega} \nu(x) \cdot Q(x) \dd{A(x)} = 0.
\end{equation}
Multiplying \cref{eq:unrelFlux} by $\nu$ and using $\nu \cdot Q = 0$ yields $0 = \nu \cdot (-\sigma \nabla u)$, which is equivalent to the Neumann condition $\partial_\nu u = 0$.

Diffusion models obtained in this way have the property of \emph{infinite propagation speed}, meaning that if initially the quantity is solely concentrated in a subdomain of $\Omega$ and zero everywhere else, for any positive time, no matter how small, the concentration is nonzero over the whole of $\Omega$. While this is acceptable for e.g. heat dissipation, this property is unrealistic when the quantity $u$ describes things like moving populations that clearly move at finite speeds. A common way to bypass this problem is to introduce a delay, which amounts to modifying the flux equation from \cref{eq:unrelFlux} to \begin{equation}
	Q(t + \tau) = - \sigma \nabla u(t),
\end{equation}
where $\tau > 0$ is the time the quantity needs to perceive the gradient and move accordingly. Unfortunately, it has been shown (cf. \cite{DreherQRacke}, or \cite{RackeDelay} for a more complex model) that this model is not well-posed. However, by formally applying the first-order Taylor expansion $Q(t + \tau) \approx Q(t) + \tau Q_t(t)$, together with \cref{eq:ContEq} the complete model becomes \begin{equation}\label{eq:telegraphersEquation}
	\left\{\begin{aligned}
		u_t + \div Q &= f\\
		(1 + \tau \partial_t) Q &= -\sigma \nabla u
	\end{aligned}\right. \quad \Rightarrow \quad (1 + \tau \partial_t) (u_t - f) = -(1 + \tau \partial_t) \div Q = \sigma \laplace u,
\end{equation}
which is a damped wave equation, or telegrapher's equation. This now hyperbolic equation can be shown to be well-posed and in addition to possess a finite speed of propagation (cf. \cite[Chapter 3]{RackeNonlinear}). This property is favorable in the kind of model we consider in this work. As in \cref{eq:FluxBoundary}, we require the boundary condition $\nu \cdot Q = 0$, then since also $\nu \cdot Q_t = 0$ we can conclude that \begin{equation}
	0 = \nu \cdot (1 + \tau \partial_t) Q = -\sigma \nu \cdot \nabla u = - \sigma \partial_\nu u,
\end{equation}
and we again have Neumann boundary conditions.

If we want to apply the relaxation from \cref{eq:telegraphersEquation} to \cref{eq:IntroState}, we have to take into account that the population ages with time. Hence, $\tau$ has to be introduced in both the time and the age variable, and a formal Taylor equation yields \cref{eq:RelInit}: \begin{equation} 
	(1 + \tau \delta) \big(\delta y(t, a, x) + L(a, x) y(t, a, x) + \Lambda(a, x, y(t, \cdot)) y(t, a, x) \big) = \sigma(a) \laplace y(t, a, x).
\end{equation}
Since this is now a second-order equation, we need additional initial/boundary conditions where $t$ or $a$ is zero. For the boundary at $t = 0$, we can simply impose another initial condition on $\delta y$. The $a = 0$-boundary is more complex because of the implicit condition \begin{equation} \label{eq:B0}
	y(t, a = 0, x) = \int_0^\amax \beta(\alpha, x) y(t, \alpha, x) \dd{\alpha}
\end{equation} we saw in \cref{eq:IntroBirth}. \textcolor{black}{In the unrelaxed equation, however, applying the Laplace operator on \cref{eq:B0} and using \cref{eq:IntroState} yields \begin{align}
	\MoveEqLeft (\delta y + L(a, x) y + \Lambda(a, x, y) y)(t, a = 0, x) = \sigma(0) \laplace \int_0^\amax \beta(\alpha, x) y(t, \alpha, x) \dd{\alpha}\\
	&= \sigma(0) \int_0^\amax \laplace \beta(\alpha, x) y(t, \alpha, x) + 2 \nabla \beta(\alpha, x) \cdot \nabla y(t, \alpha, x) + \beta(\alpha, x) \laplace y(t, \alpha, x) \dd{\alpha}\\
	&= \sigma(0) \int_0^\amax \laplace \beta(\alpha, x) y(t, \alpha, x) + 2 \nabla \beta(\alpha, x) \cdot \nabla y(t, \alpha, x)\\
	&\quad + \beta(\alpha, x) \sigma(\alpha)\inv (\delta y + L(\alpha, x) y + \Lambda(\alpha, x, y) y)(t, \alpha, x) \dd{\alpha},
	\end{align}
and plugging in \cref{eq:B0} again yields after some restructuring \begin{align}
	(\delta y)(a = 0) &= \int_0^\amax \qty(\sigma(0) \laplace \beta(\alpha, x) + \sigma(0) \beta(\alpha, x) \sigma(\alpha)\inv L(\alpha, x) - L(0, x) \beta(\alpha, x)) y(t, \alpha, x)\\
	&\quad + 2 \sigma(0) \nabla \beta(\alpha, x) \cdot \nabla y(t, \alpha, x) + \sigma(0) \beta(\alpha, x) \sigma(\alpha)\inv (\delta y)(t, \alpha, x)\\
	&\quad + \qty(\sigma(0) \beta(\alpha, x) \sigma(\alpha)\inv \Lambda(\alpha, x, y) - \Lambda(0, x, y) \beta(\alpha, x)) y(t, \alpha, x) \dd{\alpha}. \label{eq:horridCondition}
	\end{align}
Let \begin{equation}
	\begin{aligned}
		\beta_1(\alpha, x) &\coloneqq \sigma(0) \beta(\alpha, x) \sigma(\alpha)\inv,\\
		\beta_L(\alpha, x) &\coloneqq \sigma(0) \laplace \beta(\alpha, x) + \beta_1(\alpha, x) L(\alpha, x) - L(0, x) \beta(\alpha, x),\\
		\beta_\nabla(\alpha, x) &\coloneqq 2 \sigma(0) \nabla \beta(\alpha, x),
	\end{aligned} \label{eq:shorthands}
\end{equation} then we can shorten the above expression to \begin{align}
	(\delta y)(a = 0) &= \int_0^\amax \beta_L(\alpha, x) y(t, \alpha, x) + \beta_\nabla(\alpha, x) \cdot \nabla y(t, \alpha, x) + \beta_1(\alpha, x) (\delta y)(t, \alpha, x)\\
	&\quad + \qty(\beta_1(\alpha, x) \Lambda(\alpha, x, y) - \Lambda(0, x, y) \beta(\alpha, x)) y(t, \alpha, x) \dd{\alpha}.
	\end{align}
This condition holds, at least formally, for the unrelaxed model, and it turns out to be a suitable choice for a boundary condition for the relaxed model. If we introduce additional terms $g_0$, $\tilde{g}_1$ in \cref{eq:B0} resp. \cref{eq:horridCondition}, for example to account for immigration or comparison between models, we arrive at \begin{equation}
	y(t, a = 0, x) = \int_0^\amax \beta(\alpha, x) y(t, \alpha, x) \dd{\alpha} + g_0(t, x),
	\end{equation}
and \cref{eq:horridCondition} turns into \begin{align}
	(\delta y)(a = 0) &= \int_0^\amax \qty(\sigma(0) \laplace \beta(\alpha, x) + \sigma(0) \beta(\alpha, x) \sigma(\alpha)\inv L(\alpha, x) - L(0, x) \beta(\alpha, x)) y(t, \alpha, x)\\
	&\quad + 2 \sigma(0) \nabla \beta(\alpha, x) \cdot \nabla y(t, \alpha, x) + \sigma(0) \beta(\alpha, x) \sigma(\alpha)\inv (\delta y)(t, \alpha, x)\\
	&\quad + \qty(\sigma(0) \beta(\alpha, x) \sigma(\alpha)\inv \Lambda(\alpha, x, y) - \Lambda(0, x, y) \beta(\alpha, x)) y(t, \alpha, x) \dd{\alpha}\\
	&\quad + \sigma(0) \laplace g_0(t, x) - L(0, x) g_0(t, x) - \Lambda(0, y) g_0(t, x) + \tilde{g}_1(t, x),
	\end{align}
or \begin{align}
	(\delta y)(a = 0) &= \int_0^\amax \beta_L(\alpha, x) y(t, \alpha, x) + \beta_\nabla(\alpha, x) \cdot \nabla y(t, \alpha, x) + \beta_1(\alpha, x) (\delta y)(t, \alpha, x)\\
	&\quad + \qty(\beta_1(\alpha, x) \Lambda(\alpha, x, y) - \Lambda(0, x, y) \beta(\alpha, x)) y(t, \alpha, x) \dd{\alpha} - \Lambda(0, y) g_0(t, x) + g_1(t, x)
	\end{align}
with the conventions of \cref{eq:shorthands} and \begin{equation}
	g_1(t, x) \coloneqq \sigma(0) \laplace g_0(t, x) - L(0, x) g_0(t, x) + \tilde{g}_1(t, x).
	\end{equation}
For greater generality, we will assume this boundary condition for arbitrary functions $\beta_1$, $\beta_L$, $\beta_\nabla$ for the unrelaxed equation. We note that if $\beta$ does not depend on $x$, \cref{eq:horridCondition} can be stated in the symmetric form \begin{equation}
	(\delta y + L y + \Lambda(y) y) (t, a = 0, x) = \int_0^\amax \beta_1(\alpha, x) (\delta y + L y + \Lambda(y) y)(t, \alpha, x) \dd{\alpha}.
	\end{equation}}
\section{Existence of solutions} \label{sec:Existence}					
\subsection{The linearized equation with fixed birth numbers} \label{sec:RelLinFix}
\begin{subequations}\label{eq:RelLinEqn}
	This section is devoted to a proof of existence and uniqueness of solutions to \cref{eq:relaxedEquation}. Similar to the procedure in the work \cite{AzmiSchlosser} which discussed the model where $\tau = 0$, we start by considering a linearized, inhomogeneous version of \cref{eq:relaxedEquation} given by \begin{equation+}
		(1 + \tau \delta) (\delta y(t, a, x) + L(a, x) y(t, a, x)) = \sigma(a) \laplace y(t, a, x) + f(t, a, x) 
	\end{equation+}
	with the usual boundary conditions \begin{equation+}
		\quad \partial_\nu y(x \in \partial \Omega) = 0,
	\end{equation+}
	the desired initial values \begin{equation+}
		y(t = 0) = y_0, \quad \delta y(t = 0) = y_1,
	\end{equation+}
	and explicit birth numbers \begin{equation+}
		y(a = 0) = B_0, \quad \delta y(a = 0) = B_1.
	\end{equation+}
\end{subequations}
Let $t_0 \in (-\amax, T)$. By introducing characteristics and letting $v \coloneqq \restr{y}_{\chak(t_0)}$, we obtain the second-order system \begin{equation}\label{eq:RelCharEqn}
	\begin{split}
		(1 + \tau \partial_h) (v_h(h, x) + L(h, x) v(h, x)) &= \sigma(h) \laplace v(h, x) + f(t_0 + h, h, x) \text{ in } [(-t_0)^+, \amax] \times \Omega,\\
		v(h = (-t_0)^+) &= \begin{cases}
			y_0(-t_0), & t_0 < 0\\
			B_0(t_0), & t_0 > 0
		\end{cases} \eqqcolon v_0 \text{ in } \Omega,\\
		v_h(h = (-t_0)^+) &= \begin{cases}
			y_1(-t_0), & t_0 < 0\\
			B_1(t_0), & t_0 > 0
		\end{cases} \eqqcolon v_1 \text{ in } \Omega,\\
		\partial_\nu v(h) &= 0 \text{ in } [(-t_0)^+, \amax] \times \partial\Omega.
	\end{split}
\end{equation}
An application of the chain rule yields \begin{equation} \label{eq:chainRule}
	(1 + \tau \partial_h)(v_h + L(h) v) = \tau v_{hh} + (1 + \tau L(h)) v_h + (L(h) + \tau L_h(h)) v,
\end{equation}
which indicates that we need to assume that $L$ is at least once continuously differentiable. 
\begin{Ann}\label{ass:CharRelax}
	We assume \begin{enumerate}
		\item $\tau \in (0, \tau_{\text{max}})$ for some $0 <\tau_{\text{max}} < \infty$.
		\item $y_0 \in \vilbert$, $y_1 \in \hilbert$, $B_0 \in L^2((0, T), V)$, $B_1 \in L^2((0, T), H)$ and $f \in L^2((0, T), \hilbert)$.
		\item $L \in C^1(\bar{\mathcal{I}}, L^\infty(\Omega))^{n \times n}$ and its entries are uniformly bounded away from zero with respect to both age and space.
		\item $\sigma \in C^1(\bar{\mathcal{I}}, \R)^{n \times n}$ is a diagonal matrix whose entries are uniformly bounded away from zero with respect to age.
	\end{enumerate}
\end{Ann}
With \cref{eq:chainRule} and the weak formulation of the Laplace operator with Neumann boundary conditions, we arrive at the following weak formulation for \cref{eq:RelCharEqn}: for all $w \in V$ we require that \begin{gather}
	\dv{h} \skp{\tau v_h(h), w}_{H} + \skp{v_h(h) + \tau L(h) v_h(h), w}_H + \skp{L(h) v(h) + \tau L_h(h) v(h), w}_H \\
	+ \skp{\sigma(h) \nabla v(h), \nabla w}_{H^d} = \skp{f(t_0 + h, h), w}_H
\end{gather}
holds for almost all $h \in ((-t_0)^+, \amax)$.

\begin{Thm}\label{thm:exRelChar}
	Under \Cref{ass:CharRelax}, for every $t_0 \in (-\amax, T)$ there exists a unique weak solution $v \in C([(-t_0)^+, \amax], V) \cap C^1([(-t_0)^+, \amax], H) \cap C^2([(-t_0)^+, \amax], V')$ of the relaxed equation \eqref{eq:RelCharEqn} on characteristics.
\end{Thm}
\begin{proof}
	We introduce the operators \begin{align}
		\mathcal{A}_0 &\coloneqq -\sigma \laplace,& \mathcal{A}_1 &\coloneqq L + \tau L_h,&	\mathcal{B} &\coloneqq I + \tau L,& \mathcal{C} &\coloneqq \tau I,
	\end{align}
	where $I$ denotes the identity operator. Note that none of these operators depends on $t_0$. The assumptions ensure that $v_0 \in V$, $V_1 \in H$ and $\restr{f}_{\chak(t_0)} \in L^2(H)$. Further, they show that $\mathcal{A}_0 \in C^1([0, \amax], L(V, V'))$, $\mathcal{A}_1 \in C([0, \amax], L(H, H))$ (which is stronger than necessary), and $\mathcal{B}$, $\mathcal{C} \in C^1([0, \amax], L(H, H))$. Furthermore, it is easy to verify that the form generated by $\mathcal{A}_0$ is hermitian and coercive. Further, we note that $\skp{\mathcal{C}(t) u, u}_H = \tau \norm{u}_H^2$, and the form generated by $\mathcal{C}$ is obviously hermitian. Now the claim follows by \cite[Thm. XVIII.5.3 and 5.4]{DautrayLions}.
\end{proof}

It will become very important later, in particular when considering convergence issues in \Cref{thm:tauConvergence}, to precisely track the $\tau$-dependence in energy estimates. The following lemma, which can be proved using \cite[Lem. XVIII.5.7]{DautrayLions}, is the basis to this.
\begin{Lem}\label{Lem:CharEstimate}
	Let $s \coloneqq (-t_0)^+$ and $\tau_{\text{max}} > 0$. Then for all $h \in [s, \amax]$ and all $\tau \in (0, \tau_{\text{max}})$ the solution $v$ from \Cref{thm:exRelChar} satisfies the estimate \begin{equation}
		\tau \norm{v_h(h)}_H^2 + \norm{v(h)}_V^2 \leqc \tau \norm{v_1}_H^2 + \norm{v_0}_V^2 + \norm{f}_{L^2((s, \amax), H)}^2
	\end{equation}
	with a constant that does not depend on $\tau$ or $s$, but possibly on $\tau_{\text{max}}$.
\end{Lem}

\begin{Def}
	Let \begin{itemize}
		\item $M(t, s) v_0$ the value $v(t)$ of the solution $v$ to \cref{eq:RelCharEqn} with initial conditions $v(s) = v_0$, $v_h(s) = 0$ and $f \equiv 0$.
		\item $N(t, s) v_1$ the value $v(t)$ of the solution $v$ to \cref{eq:RelCharEqn} with initial conditions $v(s) = 0$, $v_h(s) = v_1$ and $f \equiv 0$.
		\item $P(t, s) f$ the value $v(t)$ of the solution $v$ to \cref{eq:RelCharEqn} with initial conditions $v(s) = 0$, $v_h(s) = 0$.
	\end{itemize}
	By $M'$, $N'$, $P'$, we denote the respective time derivative. Note that the names $M$ and $N$ are standard, cf. \cite[p. 383]{EngelNagel}, \cite{MelnikovaFilinkov} or \cite{Zheng}.
\end{Def}
\begin{Kor}\label{lem:MNestimate}
	Let $s \in (0, \amax)$, $v_0 \in V$, $v_1 \in H$, $f \in L^2((s, \amax), H)$. Then we have \begin{align}
		&M(\cdot, s) v_0 \in C([s, \amax], V) \cap C^1([s, \amax], H), \quad \norm{M(t, s) v_0}_V^2 + \tau \norm{M'(t, s) v_0}_H^2 \leqc \norm{v_0}_V^2,\\*
		&N(\cdot, s) v_1 \in C([s, \amax], V) \cap C^1([s, \amax], H), \quad \norm{N(t, s) v_1}_V^2 + \tau \norm{N'(t, s) v_1}_H^2 \leqc \tau \norm{v_1}_H^2,\\*
		&P(\cdot, s) f \in C([s, \amax], V) \cap C^1([s, \amax], H), \quad \norm{P(t, s) f}_V^2 + \tau \norm{P'(t, s) f}_H^2 \leqc \norm{f}_{L^2((s, \amax), H)}^2,
	\end{align}
	and all constants can be chosen independently of $\tau$, $t$ or $s$.
\end{Kor}
\begin{proof}
	This is a direct consequence of \Cref{Lem:CharEstimate}.
\end{proof}

Now we can reassemble the solution on the various characteristics. This gives the expression \begin{equation}\label{eq:ReprRelSolution}
	y(t, a) = \begin{cases}
		M(a, a-t) y_0(a-t) + N(a, a-t) y_1(a-t) + P(a, a-t) \restr{f}_{\chak(t-a)}, & t \leq a,\\
		M(a, 0) B_0(t-a) + N(a, 0) B_1(t-a) + P(a, 0) \restr{f}_{\chak(t-a)}, & t \geq a.
	\end{cases}
\end{equation}
It can be shown that this expression is in fact well-defined and independent of the choice of representatives for $f$, $y_0$, $y_1$, $B_0$ and $B_1$.

The next step is to verify that the function $y$ defined by \cref{eq:ReprRelSolution} is indeed a weak solution to \cref{eq:RelLinEqn}. To this end, we define a new function $v$ by \begin{equation} \label{eq:vEquation}
	v(t, a) = \begin{cases}
		M'(a, a-t) y_0(a-t) + N'(a, a-t) y_1(a-t) + P'(a, a-t) \restr{f}_{\chak(t-a)}, & t \leq a,\\
		M'(a, 0) B_0(t-a) + N'(a, 0) B_1(t-a) + P'(a, 0) \restr{f}_{\chak(t-a)}, & t \geq a.
	\end{cases}
\end{equation}

\begin{Thm} \label{thm:RelSolution}
	Assume that \Cref{ass:CharRelax} holds and define $y$ by \cref{eq:ReprRelSolution} and $v$ by \cref{eq:vEquation}. Then we have $y \in L^\infty((0, T), \vilbert)$, $v \in L^\infty((0, T), \hilbert)$, and the estimate
	\begin{equation}
		\tau \norm{v(t)}_\hilbert^2 + \norm{y(t)}_\vilbert^2 \leqc \norm{B_0}_{L^2((0, t), V)}^2 + \tau \norm{B_1}_{L^2((0, t), H)}^2 + \norm{y_0}_\vilbert^2 + \tau \norm{y_1}_\hilbert^2 + \norm{f}_{L^2((0, t), \hilbert)}^2
	\end{equation}
	for all $t \in [0, T]$.
\end{Thm}
\begin{proof}
	Using the representations of $y$ and $v$ and Fubini's theorem, we calculate \begin{align}
		\MoveEqLeft \tau \norm{v(t)}_{\hilbert}^2 + \norm{y(t)}_{\vilbert}^2 = \int_0^\amax \tau \norm{v(t, a)}_H^2 + \norm{y(t, a)}_V^2 \dd{a}\\
		&\leqc \int_0^{\min(t, \amax)} \norm{B_0(t-a)}_V^2 + \tau \norm{B_1(t-a)}_H^2 \dd{a} + \int_{\min(t, \amax)} ^\amax \norm{y_0(a-t)}_V^2 + \tau \norm{y_1(a-t)}_H^2 \dd{a} \\
		&\quad + \int_0^\amax \norm{\restr{f}_{\chak(t-a)}}_{L^2(H)}^2 \dd{a}\\
		&= \norm{B_0}_{L^2(((t - \amax)^+, t), V)}^2 + \tau \norm{B_1}_{L^2(((t - \amax)^+, t), H)}^2 + \norm{y_0}_{L^2((-(t-\amax)^+, \amax - t), V)}^2\\
		&\quad + \tau \norm{y_1}_{L^2((-(t-\amax)^+, \amax - t), H)}^2 + \int_0^\amax \int_{(a-t)^+}^a \norm{f(t-a+r, r)}_{H}^2 \dd{r} \dd{a}\\
		&\leq \norm{B_0}_{L^2((0, t), V)}^2 + \tau \norm{B_1}_{L^2((0, t), H)}^2 + \norm{y_0}_{\vilbert}^2 + \tau \norm{y_1}_{\hilbert}^2\\
		&\quad + \int_0^\amax \int_r^{\min(r+t, \amax)} \norm{f(t-a+r, r)}_{H}^2 \dd{a} \dd{r}
	\end{align}
	With the substitution $s = a-r$ we obtain
	\begin{align}
		\MoveEqLeft \tau \norm{v(t)}_{\hilbert}^2 + \norm{y(t)}_{\vilbert}^2 \leqc \norm{B_0}_{L^2((0, t), V)}^2 + \tau \norm{B_1}_{L^2((0, t), H)}^2 + \norm{y_0}_{\vilbert}^2 + \tau \norm{y_1}_{\hilbert}^2\\
		&\quad + \int_0^\amax \int_0^{\min(t, \amax - r)} \norm{f(t-s, r)}_{H}^2 \dd{s} \dd{r}\\
		& \leq \norm{B_0}_{L^2((0, t), V)}^2 + \tau \norm{B_1}_{L^2((0, t), H)}^2 + \norm{y_0}_{\vilbert}^2 + \tau \norm{y_1}_{\hilbert}^2 + \norm{f}_{L^2((0, t), \hilbert)}^2.
	\end{align}
	This concludes the proof.
\end{proof}

\begin{Thm}\label{Thm:RelSolCont}
	Assume that \Cref{ass:CharRelax} holds. Then the function $v$ from \cref{eq:vEquation} equals the weak derivative $\delta y$, where $y$ is defined by \cref{eq:ReprRelSolution}. Furthermore we have \begin{align}
		y &\in C([0, T], \vilbert) \cap C(\bar{\mathcal{I}}, L^2((0, T), V)),\\
		\delta y &\in C([0, T], \hilbert) \cap C(\bar{\mathcal{I}}, L^2((0, T), H))
	\end{align} and the energy estimate
	\begin{align}
		\tau \norm{\delta y}_{C([0, T], \hilbert)}^2 + \norm{y}_{C([0, T], \vilbert)}^2 &\leqc \norm{B_0}_{L^2((0, T), V)}^2 + \tau \norm{B_1}_{L^2((0, T), H)}^2\\
		& + \norm{y_0}_\vilbert^2 + \tau \norm{y_1}_\hilbert^2 + \norm{f}_{L^2((0, T), \hilbert)}^2. \label{eq:RelLinEstimate}
	\end{align}
	Also we have \begin{equation}
		y(t = 0) = y_0, \quad \delta y(t = 0) = y_1, \quad y(a = 0) = B_0, \quad \delta y(a = 0) = B_1.
	\end{equation}
	
\end{Thm}
\begin{proof}
	By replacing $U$ with $M$ or $N$, and $S$ with $P$, respectively, proving that $v$ is the weak derivative of $y$ can be done in a fashion similar to \cite[Thm. 2.8]{AzmiSchlosser}. The energy estimate then follows directly from \Cref{thm:RelSolution}. Because of \Cref{lem:MNestimate}, a proof similar to \cite[Thm. 2.6]{AzmiSchlosser} with the same replacements for $U$ and $S$ shows the continuity of $y$ and $v = \delta y$ and confirms the initial values for $t$ and $a$.
\end{proof}
Combining \Cref{Thm:RelSolCont} with the definition of a weak solution on characteristics allows us to conclude the following existence theorem for the linearized system \eqref{eq:RelLinEqn}.
\begin{Thm} \label{thm:RelLinWeakSol}
	Assume that \Cref{ass:CharRelax} holds. Then the function $y$ defined by \cref{eq:ReprRelSolution} is a weak solution of \cref{eq:RelLinEqn} in the sense that for almost all $t \in (0, T)$, $a \in \mathcal{I}$ and all $v \in V$ the weak formulation \begin{gather}
		\skp{\tau \delta^2 y(t, a), v}_{V' \times V} + \skp{(1 + \tau L(a)) \delta y(t, a), v}_H + \skp{(L(a) + \tau L_a(a)) y(t, a), v}_H\\
		= \skp{\sigma(a) \nabla y(t, a), \nabla v}_H + \skp{f(t, a), v}_H
	\end{gather}
	and the initial conditions \begin{alignat}{2}
		y(0, a, x) &= y_0(a, x),\quad & \delta y(0, a, x) &= y_1(a, x),\\
		y(t, 0, x) &= B_0(t, x),\quad & \delta y(t, 0, x) &= B_1(t, x)
	\end{alignat}
	hold, the latter expressions being defined by \Cref{Thm:RelSolCont}.
\end{Thm}

\subsection{The linearized equation with implicit birth laws}
				
The next step is to incorporate implicit birth conditions into model \eqref{eq:RelLinEqn}. We leave \cref{eq:RelBirth1} unchanged, and instead of the nonlinear condition \eqref{eq:RelBirth2}, we assume that the value of $\delta y$ in $a = 0$ depends implicitly on both $y$ and $\delta y$. This yields the system \begin{subequations}\label{eq:linImplSystem}
	\begin{gather+}
		(1 + \tau \delta) (\delta y(t, a, x) + L(a, x) y(t, a, x)) = \sigma(a) \laplace y(t, a, x) + f(t, a, x), \label{eq:linImplEquation}\\ 
		y(t = 0) = y_0, \quad \delta y(t = 0) = y_1,  \quad \partial_\nu y(x \in \partial \Omega) = 0, \label{eq:linImplInit}\\
		y(a = 0) = \int_0^\amax \beta_0(\alpha) y(\alpha) \dd{\alpha} + g_0 \eqqcolon B_0(t), \label{eq:BirthLawAllg0}\\
		(\delta y)(a = 0) = \int_0^\amax \beta_L(\alpha) y(\alpha) + \beta_1(\alpha) \delta y(\alpha) \dd{\alpha} \textcolor{black}{+ \beta_\nabla(\alpha) \cdot \nabla y(\alpha)} + g_1 \eqqcolon B_1(t) \label{eq:BirthLawAllg1}
	\end{gather+}
\end{subequations}
with the given \enquote{birth rates} $\beta_0$, $\beta_1$, $\beta_L$.
	
In order to make the following easier to read, we introduce a more compact notation of our results so far. We introduce the space $\mathbb{V}^\tau \coloneqq V \times H$ with the norm \begin{equation}
	\norm{(u_0, u_1)\tran}_{\mathbb{V}^\tau} = \qty(\norm{u_0}_V^2 + \tau \norm{u_1}_H^2)^{1/2}
\end{equation} (this is obviously a Banach space) and let \begin{gather}
	Y \coloneqq \begin{pmatrix}
		y\\ \delta y
	\end{pmatrix}, \quad \mathcal{U} \coloneqq \begin{pmatrix}
		M & N\\ M' & N'
	\end{pmatrix}, \quad \mathcal{S} \coloneqq \begin{pmatrix}
		P\\ P'
	\end{pmatrix},\\
	Y_0 \coloneqq \begin{pmatrix}
		y_0\\ y_1
	\end{pmatrix}, \quad \mathcal{B} \coloneqq \begin{pmatrix}
		B_0\\ B_1
	\end{pmatrix}, \quad \mathfrak{b} \coloneqq \begin{pmatrix}
		\beta_0 & 0\\ \beta_L \textcolor{black}{+ \beta_\nabla \cdot \nabla} & \beta_1
	\end{pmatrix},
\end{gather}
then from \cref{eq:ReprRelSolution} we have \begin{equation} \label{eq:Yeq}
	Y(t, a) = \begin{cases}
		\mathcal{U}(a, a-t) Y_0(a-t), & t \leq a\\
		\mathcal{U}(a, 0) \mathcal{B}(t-a), & t \geq a
	\end{cases} \,+ \mathcal{S}(a, (a-t)^+) \restr{f}_{\chak(t-a)},
\end{equation}
where \begin{equation} \label{eq:Beq}
	\mathcal{B}(t) = \int_0^\amax \mathfrak{b}(\alpha) Y(t, \alpha) \dd{\alpha} + \pmat{g_0\\g_1}.
\end{equation}
In this context, we can reformulate \cref{eq:RelLinEstimate} to \begin{equation}\label{eq:RelLinSystem}
	\norm{Y}_{C([0, T], L^2(\mathcal{I}, \mathbb{V}^\tau))}^2 \leqc \norm{\mathcal{B}}_{L^2((0, T), \mathbb{V}^\tau)}^2 + \norm{Y_0}_{L^2(\mathcal{I}, \mathbb{V}^\tau)}^2 + \norm{f}_{L^2((0, T), \hilbert)}^2.
\end{equation}
Plugging \cref{eq:Yeq} into \cref{eq:Beq} yields \begin{equation} \label{eq:VolterraRel}
	\begin{split}
		\mathcal{B}(t) &= \int_0^{\min(t, \amax)} \mathfrak{b}(\alpha) \mathcal{U}(\alpha, 0) \mathcal{B}(t-\alpha) \dd{\alpha} + \int_{\min(t, \amax)}^\amax \mathfrak{b}(\alpha) \mathcal{U}(\alpha, \alpha-t) Y_0(\alpha-t) \dd{\alpha}\\
		&\quad+ \int_0^\amax \mathfrak{b}(\alpha) \mathcal{S}(\alpha, (\alpha-t)^+) \restr{f}_{\chak(t-\alpha)} \dd{\alpha} + \pmat{g_0\\g_1}.
	\end{split}
\end{equation}

In contrast to the unrelaxed problem, this time we have to solve a Volterra equation of a function that in the first component takes values in $V$. This means that multiplication of an element $v \in V$ with $\beta_0(a)$ ($a \in \mathcal{I}$) must again yield an element of $V$. The following lemma tells us for which $\beta_0$ this holds.
\begin{Lem}\label{Lem:SobolevMultiplier}
	Let $\phi \in H^1(\Omega)$ where $\Omega \subset \R^d$ is an open and bounded domain and $d \geq 3$. Then for any $\rho \in W^{1, d}_b(\Omega) := W^{1, d}(\Omega) \cap L^\infty(\Omega)$ we have $\rho \phi \in H^1(\Omega)$ and $\norm{\rho \phi}_{H^1} \leq \norm{\phi}_{H^1} \norm{\rho}_{W^{1, d}_b(\Omega)}$. The space $W^{1, d}_b(\Omega)$ is complete with respect to the norm $\norm{\rho}_{W^{1, d}_b(\Omega)} = \norm{\rho}_{L^\infty(\Omega)} + \norm{\nabla \rho}_{L^d(\Omega)^d}$. If $d = 2$, the same results hold if we replace the space $W^{1, d}_b$ by $W^{1, d+\epsilon}_b \coloneqq W^{1, d+\epsilon}(\Omega) \cap L^\infty(\Omega)$ for some $\epsilon > 0$.
\end{Lem}
\begin{proof}
	First let $d > 2$. If $\rho \phi$ is an element of $H^1(\Omega)$ then the product rule $\nabla (\rho \phi) = \phi \nabla \rho + \rho \nabla \phi$ holds. This means that in order to show the claim we need to show that both summands are contained in $L^2(\Omega)$ again. The second one does so provided $\rho \in L^\infty$. The first one, $\phi \nabla \rho$, needs a little more reasoning. According to case Sobolev's Embedding Theorem (cf. \cite[Thm. 4.12]{AdamsFournier}),  we have $H^1(\Omega) \hookrightarrow L^q(\Omega)$ for any $2 \leq q \leq \frac{2d}{d-2}$. For these $q$ , Hölder's inequality yields ($i = 1, \ldots, d$)\begin{equation} \label{eq:SobolevMultiplier}
		\norm{\phi \cdot \partial_i \rho}^2_{L^2} = \norm{\phi^2 (\partial_i \rho)^2}_{L^1} \leq \norm{\phi^2}_{L^{q/2}} \norm{(\partial_i\rho)^2}_{L^{(q/2)'}},
	\end{equation}
	where $(q/2)'$ denotes the conjugate exponent to $q/2$, i.e. it holds $\frac{1}{q/2} + \frac{1}{(q/2)'} = 1$. In the special case $q = \frac{2d}{d-2} \geq 2$ we have $\qty(\frac q2)' = \frac d2$ and therefore $\norm{\phi \cdot \partial_i \rho}_{L^2}^2 \leq \norm{\partial_i \rho}_{L^d}^2 \norm{\phi}_{L^q}^2$. Thus, we have shown that $\nabla (\rho \phi) \in L^2(\Omega)$, and together with $\rho \in L^\infty(\Omega)$ it follows that $\rho \phi \in H^1(\Omega)$. Finally, the completeness of the space follows from $L^\infty(\Omega) \hookrightarrow L^d(\Omega)$. 
	
	If $d = 2$, Sobolev's Theorem shows that $H^1(\Omega)$ embeds into $L^q(\Omega)$ for all $q \in [2, \infty)$. Hence, the above proof works if for $q$ in \cref{eq:SobolevMultiplier} we choose the value $q = 2 + \frac 4\epsilon$. In this case we have $\qty(\frac q2)' = \frac{d + \epsilon}{2}$, and we can conclude as in the case where $d > 2$.
\end{proof}
\begin{Bem}
	From Sobolev's Theorem it follows that for any $q > d$ there is an embedding $W^{1, q}(\Omega) \hookrightarrow W^{1, d}_b(\Omega)$.
\end{Bem}

\begin{Ann}\label{ass:ImplRelax}
	For the following section we assume that \begin{enumerate}
		\item $\beta_0 \in L^\infty(\mathcal{I}, W^{1, d}_b(\Omega))$ resp. in $W^{1, d+\epsilon}_b$ for some $\epsilon > 0$ if $d = 2$, and $\beta_L$, \textcolor{black}{$\beta_\nabla$}, $\beta_1 \in L^\infty(\mathcal{I} \times \Omega)$.
		\item $g_0 \in L^2((0, T), V)$, $g_1 \in L^2((0, T), H)$.
		\item The rest of \Cref{ass:CharRelax} holds: \begin{enumerate}[{3.}1.]
			\item $\tau \in (0, \tau_{\text{max}})$ for some $0 <\tau_{\text{max}} < \infty$.
			\item $y_0 \in \vilbert$, $y_1 \in \hilbert$ and $f \in L^2((0, T), \hilbert)$
			\item $L \in C^1(\bar{\mathcal{I}}, L^\infty(\Omega))^{n \times n}$ and its entries are uniformly bounded away from zero with respect to both age and space.
			\item $\sigma \in C^1(\bar{\mathcal{I}}, \R)^{n \times n}$ is a diagonal matrix whose entries are uniformly bounded away from zero with respect to age.
		\end{enumerate}
	\end{enumerate}
\end{Ann}

By standard results on Volterra equations, especially \cite[Cor. 0.2]{Pruess}, we can solve the Volterra equation for $\mathcal{B}$:
\begin{Thm} \label{Thm:Volterra}
	The Volterra equation \eqref{eq:VolterraRel} has a unique solution $\mathcal{B} = \mathcal{B}(Y_0, f, g_0, g_1) \in L^2((0, T), \mathbb{V}^\tau)$ which satisfies the estimate \begin{equation} \label{eq:BBestimate}
		\norm{\mathcal{B}}_{L^2((0, T), \mathbb{V}^\tau)}^2 \leqc \norm{Y_0}_{L^2(\mathcal{I}, \mathbb{V}^\tau)}^2 + \norm{f}_{L^2((0, T), \hilbert)}^2 + \norm{g_0}_{L^2((0, T), V)}^2 + \tau \norm{g_1}_{L^2((0, T), H)}^2,
	\end{equation}
	the constant not depending on $\tau$.
\end{Thm}

\begin{Thm} \label{thm:RelImplExistence}
	There exists a unique weak solution $y \in C([0, T], \vilbert) \cap C(\bar{\mathcal{I}}, L^2((0, T), V))$ with $\delta y \in C([0, T], \hilbert) \cap C(\bar{\mathcal{I}}, L^2((0, T), H))$, in the sense given in \Cref{thm:RelLinWeakSol}, to the system \eqref{eq:linImplSystem}. This solution satisfies \begin{equation} \label{eq:NonlEst2*}
		\begin{aligned}
			\tau \norm{\delta y}_{C([0, T], \hilbert)}^2 + \norm{y}_{C([0, T], \vilbert)}^2 &\leqc \norm{y_0}_\vilbert^2 + \tau \norm{y_1}_\hilbert^2 + \norm{f}_{L^2((0, T), \hilbert)}^2\\
			&\quad + \norm{g_0}_{L^2((0, T), V)}^2 + \tau \norm{g_1}_{L^2((0, T), H)}^2.
		\end{aligned}
	\end{equation}
\end{Thm}
\begin{proof}
	We obtain $y$ by using the solution $\mathcal{B}$ from \Cref{Thm:Volterra} in the construction from \Cref{Thm:RelSolCont}. Plugging \cref{eq:BBestimate} into \cref{eq:RelLinSystem} yields
	\begin{equation}
		\norm{Y}_{C([0, T], L^2(\mathcal{I}, \mathbb{V}^\tau))}^2 \leqc \norm{Y_0}_{L^2(\mathcal{I}, \mathbb{V}^\tau)}^2 + \norm{f}_{L^2((0, T), \hilbert)}^2 + \norm{(g_0, g_1)}_{L^2((0, T), \mathbb{V}^\tau)}^2,
	\end{equation}
	which is just a reformulation of the claimed estimate.
\end{proof}

\begin{Bem}\label{rem:SemigroupRel}
	If $f$, $g_0$, $g_1$ are all zero and we define $\mathcal{B} = \mathcal{B}(Y_0, 0, 0, 0)$ as in \Cref{Thm:Volterra}, then one can show that $(\mathcal{T}(t))_{t \geq 0}$ defined by \begin{equation}
		(\mathcal{T}(t) Y_0)(a) \coloneqq \begin{cases}
			\mathcal{U}(a, a-t) Y_0(a-t), & t \leq a\\
			\mathcal{U}(a, 0) \mathcal{B}(t-a), & t \geq a
		\end{cases}
	\end{equation}
	is a strongly continuous semigroup of operators that formally solves \cref{eq:linImplSystem}. Formally applying Duhamel's principle shows that for any $f_1$, $f_2$, the function \begin{equation}
		V(t) = (v_1, v_2)(t) \tran \coloneqq \int_0^t \mathcal{T}(t-s) (f_1, f_2)\tran (s) \dd{s}
	\end{equation}
	yields a solution to \begin{gather}
		(1 + \tau \delta)(\delta v_1 + L v_1 - f_1) = \sigma \laplace v_1 + \tau (L f_1 + f_2),\\
		v_1(t = 0) = 0, \quad \delta v_1(t = 0) = f_1(t = 0),\\
		v_1(a = 0) = \int_0^\amax \beta_0(\alpha) v_1(\alpha) \dd{\alpha},\\
		(\delta v_1 - f_1)(a = 0) = \int_0^\amax \beta_L v_1 + \beta_\nabla \cdot \nabla v_1 + \beta_1(\delta v_1 - f_1) \dd{a}.
	\end{gather}
	The proof uses the representation \begin{align}
		V(t) = \begin{cases}
			\int_{a-t}^a \mathcal{U}(a, r) (f_1, f_2)\tran (t-a+r, r) \dd{r}, & t \leq a\\
			\mathcal{U}(a, 0) \tilde{B}(t-a) + \int_0^a \mathcal{U}(a, r) (f_1, f_2)\tran (t-a+r, r) \dd{r}, & t \geq a
		\end{cases}
	\end{align}
	where $\tilde{B}(t) = \int_0^\amax \mathfrak{b}(a) V(t, a) \dd{a}$.	The choice $f_1 = 0$, $f_2 = \tau\inv f$ gives a solution to \cref{eq:linImplSystem} without $g_0$ and $g_1$, but the extra factor of $\tau\inv$ remains present in estimates, which would yield worse convergence rates in \Cref{thm:tauConvergence}. Another choice is $f_1 = f$, $f_2 = -Lf$ and $\beta_L = \beta_1 L - L(0) \beta_0$, which captures the structure of \cref{eq:relaxedEquation} better, but has the disadvantage that it is unclear in what sense this construction actually represents a solution to the problem. In addition, the semigroup method provides less flexibility, for example it does not allow for constant terms in the birth equations. For this reason we do not pursue it any further.
\end{Bem}
												
\subsection{The nonlinear equation}\label{sec:RelNonlinear}
In this section we are ready to reintroduce the nonlinear terms to the relaxed equation to obtain model \eqref{eq:relaxedEquation} again: 
\begin{subequations}\label{eq:NonlRelSystem}
	\begin{gather+}
		(1 + \tau \delta) (\delta y + L y + \Lambda(y)) y = \sigma(a) \laplace y + f, \label{eq:NonlRelEq}\\
		y(t = 0) = y_0, \quad \delta y(t = 0) = y_1, \quad \partial_\nu y(x \in \partial \Omega) = 0, \label{eq:nonlinImplInit}\\
		y(t, a = 0) = \int_0^\amax \beta_0(\alpha, x) y(t, \alpha, x) \dd{\alpha} + g_0(t, x),\label{eq:NonlBirthLaw}\\
		\begin{aligned}
				(\delta y)(a = 0) &= \int_0^\amax \beta_L(\alpha, x) y(t, \alpha, x) + \beta_\nabla(\alpha, x) \cdot \nabla y(t, \alpha, x) + \beta_1(\alpha, x) (\delta y)(t, \alpha, x)\\
				&\quad + \qty(\beta_1(\alpha, x) \Lambda(\alpha, x, y) - \Lambda(0, x, y) \beta(\alpha, x)) y(t, \alpha, x) \dd{\alpha}\\
				&\quad - \Lambda(0, y) g_0(t, x) + g_1(t, x) \label{eq:NonlBirthLaw2}
		\end{aligned}
	\end{gather+}
\end{subequations}
Note that we keep an inhomogeneous term $f$ in \cref{eq:NonlRelEq}, for later use in \Cref{sec:Convergence}. The presence of nonlinear terms both in the equation and the implicit boundary conditions makes this section quite technical. We start with some assumptions.

\begin{Ann}\label{assump:final}
	For the rest of this chapter we assume \begin{enumerate}
		\item For all $h$, $i$, $j = 1, \ldots, n$, we have $k^{hij} \in C^1(\bar{I} \times \bar{\Omega} \times \bar{I} \times \bar{\Omega})$ and it holds $k^{hij}(a, x, \amax, \xi) = 0$ for all $a \in \mathcal{I}$ and $x$, $\xi \in \Omega$ . In other words, individuals of maximal age are not infectious at all.
		\item $g_0 \in L^\infty((0, T), H) \cap L^2((0, T), V)$.
		\item All of \Cref{ass:ImplRelax} still holds: \begin{enumerate}[{3.}1.]
			\item $\tau \in (0, \tau_{\text{max}})$ for some $0 <\tau_{\text{max}} < \infty$.
			\item $y_0 \in \vilbert$, $y_1 \in \hilbert$, $f \in L^2((0, T), \hilbert)$ and $g_1 \in L^2((0, T), H)$.
			\item $L \in C^1(\bar{\mathcal{I}}, L^\infty(\Omega))^{n \times n}$ and its entries are uniformly bounded away from zero with respect to both age and space.
			\item $\sigma \in C^1(\bar{\mathcal{I}}, \R)^{n \times n}$ is a diagonal matrix whose entries are uniformly bounded away from zero with respect to age.
			\item $\beta_0 \in L^\infty(\mathcal{I}, W^{1, d}_b(\Omega))$ resp. in $W^{1, d+\epsilon}_b$ for some $\epsilon > 0$ if $d = 2$ (for the definition of these spaces we refer to \Cref{Lem:SobolevMultiplier}) and $\beta_1 \in L^\infty(\mathcal{I} \times \Omega)$.
		\end{enumerate}
	\end{enumerate}
\end{Ann}
From \cref{eq:LambdaDefinition} we recall the definition \begin{equation}
	\Lambda(a, x, w)^{hi} = \sum_{j = 1}^{n} \int_0^\amax \int_\Omega k^{hij}(a, x, \alpha, \xi) w_j(\alpha, \xi) \dd{\xi} \dd{\alpha} = \skp{k^{hi}(a, x, \cdot, \cdot), w}_{\hilbert},
\end{equation}
where $w \in \hilbert$, $h$ and $i = 1, \ldots, n$. In what follows, to keep the notation short we will drop the indices $h$, $i$, $j$ when no confusion arises, and interpret the term $kw$ as tensor contraction. This allows us to write \cref{eq:LambdaDefinition} as \begin{equation} \label{eq:LambdaDefinition2}
	\Lambda(a, x, w) = \int_0^\amax \int_\Omega k(a, x, \alpha, \xi) w(\alpha, \xi) \dd{\xi} \dd{\alpha}.
\end{equation}
\Cref{assump:final} allows an estimate of the form \begin{equation}\label{eq:LambdaEstimateLocal}
	\norm{\Lambda(v_1) v_2}_\hilbert \leq c(k) \norm{v_1}_\hilbert \norm{v_2}_\hilbert
\end{equation}
for all $v_1$, $v_2 \in \hilbert$ and some constant $c$ depending on $k$. Integrating this estimate on $\Lambda$ with $y_1$, $y_2 \in L^2((0, T), \hilbert)$ over the interval $[0, t]$ where $t \in [0, T]$ yields
\begin{subequations}\label{eq:Lambdaestimate}
	\begin{align+}
		\norm{\Lambda(y_1) y_2}_{L^2((0, t), \hilbert)} &\leqc \norm{y_1}_{L^\infty((0, t), \hilbert)} \norm{y_2}_{L^2((0, t), \hilbert)},\\
		\intertext{or, respectively}
		\norm{\Lambda(y_1) y_2}_{L^2((0, t), \hilbert)} &\leqc \norm{y_1}_{L^2((0, t), \hilbert)} \norm{y_2}_{L^\infty((0, t), \hilbert)}
	\end{align+}
\end{subequations}
The assumptions also permit the following expression for $\Lambda(a = 0)$ we need in birth condition \eqref{eq:NonlBirthLaw2}: \begin{equation} \label{eq:LambdaA0}
	\Lambda(a = 0, x, w) = \int_0^\amax \int_\Omega k(0, x, \alpha, \xi) w(\alpha, \xi) \dd{\xi} \dd{\alpha}
\end{equation}
with a similar estimate.
	
Note that in contrast to the relaxed equation, not only does the equation contain a term $\Lambda(y)y$, but also one of the form $\tau\delta (\Lambda(y) y)$, which requires some clarification. From the definition of $\Lambda$ from \cref{eq:LambdaDefinition2} we formally have \begin{align}
	\delta(\Lambda(a, x, y(t, \cdot, \cdot)) y(t, a, x)) &= \Lambda(a, x, y(t, \cdot, \cdot)) \delta y(t, a, x)\\*
	&\quad + \Lambda_a(a, x, y(t, \cdot, \cdot)) y(t, a, x) + \Lambda(a, x, y_t(t, \cdot, \cdot)) y(t, a, x),
\end{align}
which is inconvenient since \Cref{thm:RelImplExistence} only provides information on $\delta y = y_t + y_a$, not on $y_t$ alone. However, since $\Lambda(y)$ is linear in $y$ we can write $\Lambda(y_t) = \Lambda(\delta y) - \Lambda(y_a)$, and integrating the term $\Lambda(y_a)$ by parts yields \begin{align}
	\Lambda(y_a) &= \int_\Omega \int_0^\amax k(a, x, \alpha, \xi) y_a(t, \alpha, \xi) \dd{\alpha} \dd{\xi}\\
	&= \eval[\int_\Omega k(a, x, \alpha, \xi) y(t, \alpha, \xi) \dd{\xi}|_{\alpha = 0}^{\amax} - \int_\Omega \int_0^\amax k_\alpha(a, x, \alpha, \xi) y(t, \alpha, \xi) \dd{\alpha} \dd{\xi}.
\end{align}
Because of the condition $k(a, x, \amax, \xi) = 0$, the term where we evaluate the integral at $\amax$ vanishes. (Note that this condition would not be necessary if one could show that $y(a = \amax) = 0$, as in \cite[Thm. 4.2]{AnitaArnautuCapasso}.) Using the implicit birth condition \eqref{eq:NonlBirthLaw} we can rewrite the $(\alpha = 0)$-term as \begin{align}
	\MoveEqLeft \int_\Omega k(a, x, 0, \xi) y(t, 0, \xi) \dd{\xi} = \int_\Omega k(a, x, 0, \xi) \qty(\int_0^\amax \beta_0(\alpha, \xi) y(t, \alpha, \xi) \dd{\alpha} + g_0(t, \xi)) \dd{\xi}\\
	&= \int_\Omega \int_0^\amax \big(k(a, x, 0, \xi) \beta_0(\alpha, \xi)\big) y(t, \alpha, \xi) \dd{\alpha} \dd{\xi} + \int_\Omega k(a, x, 0, \xi) g_0(t, \xi) \dd{\xi},
\end{align}
which shows that $\delta \Lambda$ will also depend on $g_0$, in a spatially nonlocal way. Hence, for $v$, $w \in L^\infty((0, T), \hilbert)$ we introduce the notation $\delta \Lambda(v, g_0) w$ defined by \begin{equation} \label{eq:deltaLambdaDef}
	\delta\Lambda(v, g_0) w(t, a, x) \coloneqq  \Lambda(v)\delta w + \Lambda(\delta v) w + \Lambda_1(v) w + \Lambda_2(g_0(t)) w,\\
\end{equation}
where, with a similar notation as in \cref{eq:LambdaDefinition2} we let ($h$, $i = 1, \ldots, n$) \begin{align}
	\Lambda_1(v)^{hi} &\coloneqq \sum_{j = 1}^{n} \int_\Omega \int_0^\amax \tilde{k}^{hij}(a, x, \alpha, \xi) v_j(\alpha, \xi) \dd{\alpha} \dd{\xi},\\
	\tilde{k}^{hij}(a, x, \alpha, \xi) &\coloneqq k_a^{hij}(a, x, \alpha, \xi) + k_\alpha^{hij}(a, x, \alpha, \xi) + \sum_{\ell = 1}^{n} k^{hi\ell}(a, x, 0, \xi) \beta_0^{\ell j}(\alpha, \xi),\\
	\Lambda_2(g_0(t))^{hi} &\coloneqq \sum_{j = 1}^{n} \int_\Omega k^{hij}(a, x, 0, \xi) (g_0)_j(t, \xi) \dd{\xi}.
\end{align}
This will be the expression we use for $\delta \Lambda$ in \cref{eq:NonlRelEq}. Furthermore we write \begin{equation} \label{eq:EinsPlusDeltaLambdaDef}
	(1 + \tau \delta)\Lambda(v, g_0) w \coloneqq \Lambda(v) w + \tau \delta \Lambda(v, g_0) w,
\end{equation}
with $\delta \Lambda$ defined as in \cref{eq:deltaLambdaDef}. In total, \cref{eq:NonlRelEq} takes the following form: \begin{equation}
	(1 + \tau \delta) (\delta y + L y) + \Lambda(y) y + \tau\Lambda(y)\delta y + \tau \Lambda(\delta y) y + \tau \Lambda_1(y) y + \tau \Lambda_2(g_0) y = \sigma \laplace y + f.
\end{equation}
\Cref{assump:final} yields that $k_a + k_\alpha + k(\alpha = 0) \beta_0 \in L^\infty(\mathcal{I} \times \Omega, \hilbert)$. Thus we can estimate \begin{align} \label{eq:deltaLambdaEstimate}
	\norm{\delta\Lambda(v(t), g_0(t)) w(t)}_\hilbert &\leqc \norm{v(t)}_\hilbert \norm{w(t)}_\hilbert + \norm{\delta v(t)}_\hilbert \norm{w(t)}_\hilbert\\
	&\quad + \norm{v(t)}_\hilbert \norm{\delta w(t)}_\hilbert + \norm{g_0(t)}_H \norm{v(t)}_\hilbert
\end{align}
This shows that, contrary to \Cref{ass:ImplRelax}, assuming $g_0 \in L^2((0, T), H)$ is not enough, we need the stronger assumption $g_0 \in L^\infty((0, T), H)$. Combining \cref{eq:LambdaEstimateLocal} with \cref{eq:deltaLambdaEstimate} and integration with respect to the time variable yields \begin{subequations}\label{eq:NonlEst1}
	\begin{equation+}
		\begin{aligned}
			\MoveEqLeft \norm{(1 + \tau\delta)\Lambda(v, g_0) w}_{L^2((0, T), \hilbert)} \leqc \norm{v}_{L^\infty((0, T), \hilbert)} \norm{w}_{L^2((0, T), \hilbert)} + \tau \norm{\delta v}_{L^\infty((0, T), \hilbert)} \norm{w}_{L^2((0, T), \hilbert)}\\
			&\quad + \tau \norm{v}_{L^\infty((0, T), \hilbert)} \norm{\delta w}_{L^2((0, T), \hilbert)} + \tau \norm{g_0}_{L^\infty((0, T), H)} \norm{w}_{L^2((0, T), \hilbert)},
		\end{aligned}
	\end{equation+}
	or alternatively \begin{equation+}
		\begin{aligned}
			\MoveEqLeft \norm{(1 + \tau\delta)\Lambda(v, g_0) w}_{L^2((0, T), \hilbert)} \leqc \norm{v}_{L^2((0, T), \hilbert)} \norm{w}_{L^\infty((0, T), \hilbert)} + \tau \norm{\delta v}_{L^2((0, T), \hilbert)} \norm{w}_{L^\infty((0, T), \hilbert)}\\
			&\quad + \tau \norm{v}_{L^2((0, T), \hilbert)} \norm{\delta w}_{L^\infty((0, T), \hilbert)} + \tau \norm{g_0}_{L^\infty((0, T), H)} \norm{w}_{L^2((0, T), \hilbert)}.
		\end{aligned}
	\end{equation+}
\end{subequations}

\textcolor{black}{WE now turn our attention to the nonlinear first-order birth equation \begin{align}
		(\delta y)(a = 0) &= \int_0^\amax \beta_L(\alpha, x) y(t, \alpha, x) + \beta_\nabla(\alpha, x) \cdot \nabla y(t, \alpha, x) + \beta_1(\alpha, x) (\delta y)(t, \alpha, x)\\
		&\quad + \qty(\beta_1(\alpha, x) \Lambda(\alpha, x, y) - \Lambda(0, x, y) \beta(\alpha, x)) y(t, \alpha, x) \dd{\alpha} - \Lambda(0, y) g_0(t, x) + g_1(t, x)
\end{align}}
from \cref{eq:NonlBirthLaw2}.
We define \begin{align}
	G(v)(w, g_0) &\coloneqq \int_0^\amax (\beta_1(\alpha) \Lambda(\alpha, v) - \Lambda(a=0, v) \beta_0(\alpha)) w(\alpha) \dd{\alpha} - \Lambda(a = 0, v) g_0 \label{eq:GDef}
\end{align} 
for $v$, $w \in L^\infty((0, T)), \hilbert$, where $\Lambda(a = 0)$ is defined as in \cref{eq:LambdaA0}. Then we interpret \cref{eq:NonlBirthLaw2} as \begin{equation} \label{eq:NonlBirthLaw1Real}
	\delta y(a = 0) = \int_0^\amax \beta_1(\alpha) \delta y(\alpha) + \beta_L(\alpha) y(\alpha) \textcolor{black}{+ \beta_\nabla(\alpha) \cdot \nabla y(\alpha)} \dd{\alpha} + G(y)(y, g_0) + g_1 
\end{equation}
From \Cref{assump:final} we can conclude that $\beta_L \in L^\infty(\mathcal{I} \times \Omega)$.
Note that $G$ from \cref{eq:GDef} is nonlocal in $v$ and linear in both of its arguments $v$ and $(w, g_0)$, hence the seemingly strange notation. We can estimate \begin{subequations}\label{eq:Gestimate}
	\begin{equation+}
		\begin{aligned}
			\MoveEqLeft \norm{G(v)(w, g_0)}_{L^2((0, T), H)} \leqc \norm{\beta_1 \Lambda(v) w}_{L^2((0, T), \hilbert)}\\
			&\quad + \norm{\Lambda(a=0, v) \beta_0 w}_{L^2((0, T), \hilbert)} + \norm{\Lambda(a = 0, v) g_0}_{L^2((0, T), H)}\\*
			&\leqc \norm{v}_{L^\infty((0, T), \hilbert)} \norm{w}_{L^2((0, T), \hilbert)} + \norm{v}_{L^2((0, T), \hilbert)} \norm{g_0}_{L^\infty((0, T), H)},
		\end{aligned}
	\end{equation+}
	or alternatively \begin{equation+}
		\norm{G(v)(w, g_0)}_{L^2((0, T), H)} \leqc \norm{v}_{L^2((0, T), \hilbert)} \norm{w}_{L^\infty((0, T), \hilbert)} + \norm{v}_{L^2((0, T), \hilbert)} \norm{g_0}_{L^\infty((0, T), H)}.
	\end{equation+}
\end{subequations}
				
Altogether, \cref{eq:NonlRelSystem} can be reformulated as \begin{equation} \label{eq:NonlRelSystem2}
	\begin{aligned}
		(1 + \tau \delta) (\delta y + L y)  &= \sigma \laplace y - \Lambda(y) y - \tau\Lambda(y)\delta y - \tau \Lambda(\delta y) y - \tau \Lambda_1(y) y - \tau \Lambda_2(g_0) y + f,\\
		y(t = 0) &= y_0, \quad \delta y(t = 0) = y_1, \quad \partial_\nu y = 0 \text{ in } \partial\Omega,\\
		y(a = 0) &= \int_0^\amax \beta_0(\alpha) y(\alpha) \dd{\alpha} + g_0,\\
		\delta y (a = 0) &= \int_0^\amax \beta_1(\alpha) \delta y(\alpha) + \beta_L(\alpha) y(\alpha) \textcolor{black}{+ \beta_\nabla(\alpha) \cdot \nabla y(\alpha)} \dd{\alpha} + G(y)(y, g_0) + g_1. 
	\end{aligned}
\end{equation}
The strategy of finding a solution to \cref{eq:NonlRelSystem2} is similar to \cite[Thm. 2.15]{AzmiSchlosser}: find fixed points of the map $\Phi$ which maps a given $w$ to the weak solution $y$ (in the sense of \Cref{thm:RelImplExistence}) of \begin{equation} \label{eq:PhiSystem}
	\begin{aligned}
		(1 + \tau \delta) (\delta y + L y)  &= \sigma \laplace y - (1 + \tau \delta) \Lambda(w, g_0) w + f,\\
		y(t = 0) &= y_0, \quad \delta y(t = 0) = y_1, \quad \partial_\nu y = 0 \text{ in } \partial\Omega,\\
		y(a = 0) &= \int_0^\amax \beta_0(\alpha) y(\alpha) \dd{\alpha} + g_0,\\
		\delta y (a = 0) &= \int_0^\amax \beta_1(\alpha) \delta y(\alpha) + \beta_L(\alpha) y(\alpha) \textcolor{black}{+ \beta_\nabla(\alpha) \cdot \nabla y(\alpha)} \dd{\alpha} + G(w)(w, g_0) + g_1. 
	\end{aligned}
\end{equation}
Using \cref{eq:NonlEst1} and \cref{eq:Gestimate} shows that under \Cref{assump:final}, for any $w \in L^\infty((0, T), \hilbert)$ with $\delta w \in L^\infty((0, T), \hilbert)$, the conditions of \Cref{thm:RelImplExistence} are satisfied. Thus, \cref{eq:NonlEst2*} together with the fact that $\tau \leq \tau_{\text{max}}$ yields an estimate for a solution of \cref{eq:PhiSystem}: \begin{align}
	\MoveEqLeft \tau \norm{\delta y(t)}_\hilbert^2 + \norm{y(t)}_\vilbert^2 \leqc \norm{y_0}_\vilbert^2 + \tau \norm{y_1}_\hilbert^2 + \norm{f}_{L^2((0, T), \hilbert)}^2 + \norm{(1 + \tau \delta)\Lambda(w, g_0) w}_{L^2((0, t), \hilbert)^2}^2\\
	&\quad + \norm{g_0}_{L^2((0, T), V)}^2 + \tau \norm{G(w)(w, g_0) + g_1}_{L^2((0, t), H)}^2\\
	&\leqc \norm{y_0}_\vilbert^2 + \tau \norm{y_1}_\hilbert^2 + \norm{g_0}_{L^2((0, T), V)}^2 + \norm{f}_{L^2((0, T), \hilbert)}^2\\
	&\quad + \Big(\norm{w}_{L^2((0, T), \hilbert)} \norm{w}_{L^\infty((0, T), \hilbert)} + \tau \norm{\delta w}_{L^2((0, T), \hilbert)} \norm{w}_{L^\infty((0, T), \hilbert)}\\
	&\quad + \tau \norm{g_0}_{L^\infty((0, T), H)} \norm{w}_{L^2((0, T), \hilbert)} \Big)^2 + \tau \Big(\norm{w}_{L^2((0, T), \hilbert)} \norm{w}_{L^\infty((0, T), \hilbert)}\\
	&\quad + \norm{w}_{L^2((0, T), \hilbert)} \norm{g_0}_{L^\infty((0, T), H)} + \norm{g_1}_{L^2((0, t), H)} \Big)^2\\
	&\leqc \norm{y_0}_\vilbert^2 + \tau \norm{y_1}_\hilbert^2 + \norm{f}_{L^2((0, T), \hilbert)}^2 + \norm{g_0}_{L^2((0, T), V)}^2 + \tau \norm{g_1}_{L^2((0, t), H)}^2\\
	&\quad + \norm{w}_{L^2((0, T), \hilbert)}^2 \norm{w}_{L^\infty((0, T), \hilbert)}^2 + \tau^2 \norm{\delta w}_{L^2((0, T), \hilbert)}^2 \norm{w}_{L^\infty((0, T), \hilbert)}^2\\
	&\quad + \tau \norm{g_0}_{L^\infty((0, T), H)}^2 \norm{w}_{L^2((0, T), \hilbert)}^2. \label{eq:NonlEst2}
\end{align}


Now we can proceed as in \cite[Thm. 14.2 and Lem. 14.3]{Smoller}.
\begin{Lem} \label{Lem:LambdaGLipschitz}
	$(1 + \tau \delta) \Lambda$ and $G$ are locally Lipschitz continuous.
\end{Lem}
\begin{proof}
	Let $v_i$, $w_i \in L^\infty((0, T), \hilbert)$ and $g_0^i \in L^2((0, T), H)$ ($i \in \set{1, 2}$). If also $\delta v_i$, $\delta w_i \in L^\infty((0, T), \hilbert)$, \cref{eq:NonlEst1} yields \begin{align}
		&\quad \norm{(1 + \tau\delta)[\Lambda(v_1, g_0^1) w_1] - (1 + \tau\delta)[\Lambda(v_2, g_0^2) w_2]}_{L^2((0, T), \hilbert)}\\
		&\leqc \norm{(1 + \tau\delta)[\Lambda(v_1, g_0^1) (w_1- w_2)]}_{L^2((0, T), \hilbert)} + \norm{(1 + \tau\delta)[\Lambda(v_1 - v_2, g_0^1 - g_0^2) w_2]}_{L^2((0, T), \hilbert)}\\
		&\leqc \norm{v_1}_{L^\infty((0, T), \hilbert)} \norm{w_1 - w_2}_{L^2((0, T), \hilbert)} + \tau \norm{\delta v_1}_{L^\infty((0, T), \hilbert)} \norm{w_1 - w_2}_{L^2((0, T), \hilbert)}\\
		&\quad + \tau \norm{v_1}_{L^\infty((0, T), \hilbert)} \norm{\delta w_1 - \delta w_2}_{L^2((0, T), \hilbert)} + \tau \norm{g_0^1}_{L^\infty((0, t), H)} \norm{w_1 - w_2}_{L^2((0, T), \hilbert)}\\
		&\quad + \norm{v_1 - v_2}_{L^2((0, T), \hilbert)} \norm{w_2}_{L^\infty((0, T), \hilbert)} + \tau \norm{\delta v_1 - \delta v_2}_{L^2((0, T), \hilbert)} \norm{w_2}_{L^\infty((0, T), \hilbert)}\\
		&\quad + \tau \norm{v_1 - v_2}_{L^2((0, T), \hilbert)} \norm{\delta w_2}_{L^\infty((0, T), \hilbert)} + \tau \norm{g_0^1 - g_0^2}_{L^\infty((0, t), H)} \norm{w_2}_{L^2((0, T), \hilbert)}\\
		&\leqc \max\set{\norm{v_1}_{L^\infty((0, T), \hilbert)}, \tau \norm{\delta v_1}_{L^\infty((0, T), \hilbert)}, \norm{w_2}_{L^\infty((0, T), \hilbert)}, \tau \norm{\delta w_2}_{L^\infty((0, T), \hilbert)}, \tau \norm{g_0^1}_{L^\infty((0, t), H)}}\\
		&\quad \cdot \Big(\norm{v_1 - v_2}_{L^2((0, T), \hilbert)} + \norm{w_1 - w_2}_{L^2((0, T), \hilbert)} + \tau \norm{\delta v_1 - \delta v_2}_{L^2((0, T), \hilbert)}\\
		&\qquad + \tau \norm{\delta w_1 - \delta w_2}_{L^2((0, T), \hilbert)} + \tau \norm{g_0^1 - g_0^2}_{L^\infty((0, T), H)} \Big). \label{eq:LambdaLipschitz}
	\end{align}
	Similarly, for $G$ we have, using \cref{eq:Gestimate}: \begin{align}
		&\quad \norm{G(v_1)(w_1, g_0^1) - G(v_2)(w_2, g_0^2)}_{L^2((0, T), H)}\\
		&\leq \norm{G(v_1 - v_2)(w_1, g_0^1)}_{L^2((0, T), H)} + \norm{G(v_2)(w_1 - w_2, g_0^1 - g_0^2)}_{L^2((0, T), H)}\\
		&\leqc \max\set{\norm{v_1}_{L^\infty((0, T), \hilbert)}, \norm{w_2}_{L^\infty((0, T), \hilbert)}, \norm{g_0^1}_{L^\infty((0, T), H)}}\\*
		&\quad \cdot \qty(\norm{v_1 - v_2}_{L^2((0, T), \hilbert)} + \norm{w_1 - w_2}_{L^2((0, T), \hilbert)} + \norm{g_0^1 - g_0^2}_{L^2((0, T), H)}). \label{eq:GLipschitz}
	\end{align}
\end{proof}

\begin{Lem}\label{lem:NonlRelGronwall}
	Let $y^1$ and $y^2$ be two solutions of \cref{eq:NonlRelSystem2} to the respective data $y_0^i$, $y_1^i$, $f^i$, $g_0^i$, $g_1^i$ ($i \in \set{1, 2}$) and let $y \coloneqq y^1 - y^2$. Further assume that $\norm{y^i}_{L^\infty((0, T), \hilbert)}$, $\tau \norm{\delta y^i}_{L^\infty((0, T), \hilbert)}$ and $\tau \norm{g_0^i}_{L^\infty((0, T), H)}$ are bounded above by some $S$ for $i \in \set{1, 2}$. Then we have an estimate of the form \begin{align}
		\tau \norm{\delta y}_{C([0, T], \hilbert)}^2 + \norm{y(t)}_{C([0, T], \vilbert)}^2 &\leqc \norm{y_0^1 - y_0^2}_\vilbert^2 + \tau \norm{y_1^1 - y_1^2}_\hilbert^2 + \norm{f^1 - f^2}_{L^2((0, T), \hilbert)}^2\\
		&\quad + \norm{g_0^1 - g_0^2}_{L^2((0, T), V)}^2 + \tau \norm{g_1^1 - g_1^2}_{L^2((0, T), H)}^2.
	\end{align}
\end{Lem}
\begin{proof}
	The difference $y$ is a solution to \begin{align}
		(1 + \tau \delta) (\delta y + L y) y &= \sigma(a) \laplace y - (1 + \tau \delta)(\Lambda(y^1, g_0^1) y^1 - \Lambda(y^2, g_0^2) y^2) + f^1 - f^2,\\
		y(t = 0) &= y_0^1 - y_0^2, \quad \delta y(t = 0) = y_1^1 - y_1^2,\\
		\partial_\nu y &= 0 \text{ in } \partial\Omega,\\
		y(a = 0) &= \int_0^\amax \beta_0(\alpha) y(\alpha) \dd{\alpha} + g_0^1 - g_0^2,\\
		(\delta y)(a = 0) &= \int_0^\amax \beta_1(\alpha) \delta y(\alpha) + \beta_L(\alpha) y(\alpha) \textcolor{black}{+ \beta_\nabla(\alpha) \cdot \nabla y(\alpha)} \dd{\alpha}\\
		&\quad + G(y^1)(y^1, g_0^1) - G(y^2)(y^2, g_0^2) + g_1^1 - g_1^2. 
	\end{align}
	By \Cref{thm:RelImplExistence} and \Cref{Lem:LambdaGLipschitz} we obtain \begin{align}
		&\quad \tau \norm{\delta y(t)}_\hilbert^2 + \norm{y(t)}_\vilbert^2\\
		&\leqc \norm{y_0^1 - y_0^2}_\vilbert^2 + \tau \norm{y_1^1 - y_1^2}_\hilbert^2 + \norm{(1 + \tau \delta)(\Lambda(y^1, g_0^1) y^1 - \Lambda(y^2, g_0^2) y^2)}_{L^2((0, t), \hilbert)}^2\\
		&\quad + \norm{f^1 - f^2}_{L^2((0, t), \hilbert)}^2 + \norm{g_0^1 - g_0^2}_{L^2((0, t), V)}^2 + \tau \norm{g_1^1 - g_1^2}_{L^2((0, t), H)}^2\\
		&\quad + \tau \norm{G(y^1)(y^1, g_0^1) - G(y^2)(y^2, g_0^2)}_{L^2((0, t), H)}^2 \\
		&\leqc \norm{y_0^1 - y_0^2}_\vilbert^2 + \tau \norm{y_1^1 - y_1^2}_\hilbert^2 + \norm{f^1 - f^2}_{L^2((0, t), \hilbert)}^2\\
		&\quad + \norm{g_0^1 - g_0^2}_{L^2((0, t), V)}^2 + \tau \norm{g_1^1 - g_1^2}_{L^2((0, t), H)}^2 + \tau \norm{g_0^1 - g_0^2}_{L^2((0, t), H)}^2\\
		&\quad + \max\set{\norm{y^1}_{L^\infty((0, t), \hilbert)}^2, \tau \norm{\delta y^1}_{L^\infty((0, t), \hilbert)}^2, \tau \norm{g_0^1}_{L^\infty((0, t), H)}^2, \norm{y^2}_{L^\infty((0, t), \hilbert)}^2, \tau \norm{\delta y^2}_{L^\infty((0, t), \hilbert)}^2}\\
		&\quad \cdot \qty(\norm{y^1 - y^2}_{L^2((0, t), \hilbert)}^2 + \tau \norm{\delta y^1 - \delta y^2}_{L^2((0, t), \hilbert)}^2 + \tau \norm{g_0^1 - g_0^2}_{L^2((0, t), H)}^2). \label{eq:RelUniqueness}
	\end{align}
	Estimating by $S$ yields
	\begin{align}
		\MoveEqLeft \tau \norm{\delta y(t)}_\hilbert^2 + \norm{y(t)}_\vilbert^2 \leqc \norm{y_0^1 - y_0^2}_\vilbert^2 + \tau \norm{y_1^1 - y_1^2}_\hilbert^2 + \norm{f^1 - f^2}_{L^2((0, t), \hilbert)}^2\\
		&\quad + \norm{g_0^1 - g_0^2}_{L^2((0, t), V)}^2 + \tau \norm{g_1^1 - g_1^2}_{L^2((0, t), H)}^2 + \tau \norm{g_0^1 - g_0^2}_{L^2((0, t), H)}^2\\
		&\quad + S^2 \qty(\norm{y^1 - y^2}_{L^2((0, T), \hilbert)} + \tau \norm{\delta y^1 - \delta y^2}_{L^2((0, T), \hilbert)} + \tau \norm{g_0^1 - g_0^2}_{L^2((0, t), H)})^2.
	\end{align}
	Now Gronwall's lemma yields the estimate \begin{align}
		\MoveEqLeft \tau \norm{\delta y(t)}_\hilbert^2 + \norm{y(t)}_\vilbert^2 \leqc \norm{y_0^1 - y_0^2}_\vilbert^2 + \tau \norm{y_1^1 - y_1^2}_\hilbert^2 + \norm{f^1 - f^2}_{L^2((0, t), \hilbert)}^2\\*
		&\quad + \norm{g_0^1 - g_0^2}_{L^2((0, t), V)}^2 + \tau \norm{g_1^1 - g_1^2}_{L^2((0, t), H)}^2 + \tau \norm{g_0^1 - g_0^2}_{L^2((0, t), H)}^2\\*
		&\quad + \tau^2 S^2 \norm{g_0^1 - g_0^2}_{L^2((0, t), H)}^2,
	\end{align}
	and estimating $\tau$ by $\tau_{\text{max}}$ and norms on $(0, t)$ by those on $(0, T)$ shows the claim.
\end{proof}

Letting $y_0^1 = y_0^2$ etc. in the previous theorem directly yields
\begin{Kor}
	There is at most one weak solution $y \in L^\infty((0, T), \hilbert)$ to the system \eqref{eq:NonlRelSystem2}.
\end{Kor}

\begin{Thm} \label{thm:RelMain}
	Under \Cref{assump:final} there exists a $T^* > 0$ such that for every $T \in (0, T^*)$ there is a weak solution $y \in C([0, T], \vilbert) \cap C(\bar{\mathcal{I}}, L^2((0, T), V))$ with $\delta y \in C([0, T], \hilbert) \cap C(\bar{\mathcal{I}}, L^2((0, T), H))$ to the system \eqref{eq:NonlRelSystem2}. That means, for all $v \in V$ the weak formulation \begin{gather}
		\skp{\tau \delta^2 y(t, a), v}_{V' \times V} + \skp{(1 + \tau L(a)) \delta y(t, a), v}_H + \skp{(L(a) + \tau L_a(a)) y(t, a), v}_H\\
		= \skp{\sigma(a) \nabla y(t, a), \nabla v}_H + \skp{f(t, a), v}_H - \skp{(1 + \tau \delta) \Lambda(y, g_0) y, v}_H,
	\end{gather} where $(1 + \tau \delta) \Lambda$ is defined as in \cref{eq:EinsPlusDeltaLambdaDef}, is satisfied and the initial conditions \eqref{eq:nonlinImplInit}, \eqref{eq:NonlBirthLaw} and \eqref{eq:NonlBirthLaw1Real} hold. In addition, $y$ satisfies the estimate \begin{equation}
		\norm{y}_{C([0, T], \vilbert)}^2 + \tau \norm{\delta y}_{C([0, T], \hilbert)}^2 \leqc \mathcal{K}
	\end{equation}
	where \begin{equation}
		\mathcal{K} \coloneqq \norm{y_0}_\vilbert^2 + \tau \norm{y_1}_\hilbert^2 + \norm{f}_{L^2((0, t), \hilbert)}^2 + \norm{g_0}_{L^2((0, t), V)}^2 + \tau \norm{g_1}_{L^2((0, T), H)}^2.
	\end{equation}
	More precisely, $T^*$ can be chosen as $C \cdot (\mathcal{K} + \tau \norm{g_0}_{L^\infty((0, T), H)}^2)\inv$ with a constant $C$ independent of $y_0$, $y_1$, $f$, $g_0$, $g_1$, or $\tau$, but possibly depending on $\tau_{\text{max}}$.
\end{Thm}
\begin{proof}	
	Let $X \coloneqq \set{u \in C([0, T], \vilbert) \where \exists \delta u \in C([0, T], \hilbert)}$ endowed with the norm \begin{equation}
		\norm{u}_{X, \tau} \coloneqq \qty(\norm{u}_{C([0, T], \vilbert)}^2 + \tau \norm{\delta u}_{C([0, T], \hilbert)}^2)^{1/2},
	\end{equation}
	then it is easy to see that $X$ is a Banach space. Further let $\Phi$ as in \cref{eq:PhiSystem}. Estimate \eqref{eq:NonlEst2} shows that $\Phi$ is a well-defined map in $X$. Furthermore let $\bar{y} \coloneqq \Phi(0)$ the solution to the linear equation
	\begin{align}
		(1 + \tau \delta) (\delta \bar{y} + L \bar{y}) y &= \sigma(a) \laplace \bar{y} + f,\\*
		\bar{y}(t = 0) &= y_0, \quad \delta \bar{y}(t = 0) = y_1,\\*
		\partial_\nu \bar{y} &= 0 \text{ in } \partial\Omega,\\*
		\bar{y}(a = 0) &= \int_0^\amax \beta_0(\alpha) \bar{y}(\alpha) \dd{\alpha} + g_0\\*
		(\delta \bar{y})(a = 0) &= \int_0^\amax \beta_1(\alpha) \delta \bar{y}(\alpha) + \beta_L(\alpha) \bar{y}(\alpha) \textcolor{black}{+ \beta_\nabla(\alpha) \cdot \nabla \bar{y}(\alpha)} \dd{\alpha} + g_1.
	\end{align}
	From \cref{eq:NonlEst2*} we have the estimate \begin{equation}
		\norm{\bar{y}}_{X, \tau}^2 \leqc \norm{y_0}_\vilbert^2 + \tau \norm{y_1}_\hilbert^2 + \norm{f}_{L^2((0, t), \hilbert)}^2 + \norm{g_0}_{L^2((0, T), V)}^2 + \tau \norm{g_1}_{L^2((0, T), H)}^2 \leqc \mathcal{K}.
	\end{equation}
	Finally let
	\begin{equation}
		\Gamma^\tau = \set{v \in X \where \norm{v - \bar{y}}_{X, \tau}^2 \leq \mathcal{K}}.
	\end{equation}
	Obviously, $\Gamma^\tau$ is a nonempty and closed set in $X$. Since for any $v \in \Gamma^\tau$ we have \begin{align}
		\norm{v}_{X, \tau}^2 &\leqc \norm{v - \bar{y}}_{X, \tau}^2 + \norm{\bar{y}}_{X, \tau}^2 \leqc \mathcal{K}\label{eq:GammaEstimate}
	\end{align}
	the set $\Gamma^\tau$ is also bounded in $X$.

	Now let $y^1$ and $y^2$ both in $\Gamma^\tau$ and let $y \coloneqq \Phi(y^2) - \Phi(y^2)$, then $y$ satisfies \begin{align}
		(1 + \tau \delta) (\delta y + L y) y &= \sigma(a) \laplace y - (1 + \tau \delta)(\Lambda(y^1, g_0) y^1 - \Lambda(y^2, g_0) y^2),\\
		y(t = 0) &= 0, \quad \delta y(t = 0) = 0,\\
		\partial_\nu y &= 0 \text{ in } \partial\Omega,\\
		y(a = 0) &= \int_0^\amax \beta_0(\alpha) y(\alpha) \dd{\alpha}\\
		(\delta y)(a = 0) &= \int_0^\amax \beta_1(\alpha) \delta y(\alpha) + \beta_L(\alpha) y(\alpha) \textcolor{black}{+ \beta_\nabla(\alpha) \cdot \nabla y(\alpha)} \dd{\alpha}\\
		&\quad + G(y^1)(y^1, g_0) - G(y^2)(y^2, g_0)
	\end{align}
	Thus, a similar calculation to \cref{eq:RelUniqueness} combined with \cref{eq:GammaEstimate} yields \begin{align}
		\norm{y}_{X, \tau}^2 &\leqc \max \left\{ \norm{y^1}_{L^\infty((0, T), \hilbert)}^2, \tau \norm{\delta y^1}_{L^\infty((0, T), \hilbert)}^2, \tau \norm{g_0}_{L^\infty((0, T), H)}^2, \norm{y^2}_{L^\infty((0, T), \hilbert)}^2,\right.\\
		&\quad \left. \tau \norm{\delta y^2}_{L^\infty((0, T), \hilbert)}^2 \right\}  \cdot \qty(\norm{y^1 - y^2}_{L^2((0, T), \hilbert)}^2 + \tau \norm{\delta y^1 - \delta y^2}_{L^2((0, T), \hilbert)}^2)\\
		&\leqc \qty(\mathcal{K} + \tau \norm{g_0}_{L^\infty((0, T), H)}^2) T \norm{y^1 - y^2}_{X, \tau}.
	\end{align}
	Choosing $T$ small enough, proportional to $\qty(\mathcal{K} + \tau \norm{g_0}_{L^\infty((0, T), H)}^2)\inv$, allows us to compensate for the constant in front of the expression, turning $\Phi$ into a contraction on $\Gamma^\tau$. Together with \cref{eq:GammaEstimate}, this also shows that for any $y \in \Gamma^\tau$, the norm of $\Phi(y) - \bar{y} = \Phi(v) - \Phi(0)$ can be bounded above by \begin{equation}
		\norm{\Phi(y) - \bar{y}}_{X, \tau} \leqc \qty(\mathcal{K} + \tau \norm{g_0}_{L^\infty((0, T), H)}^2) \cdot T \norm{y}_{X, \tau}^2 \leqc \mathcal{K} \qty(\mathcal{K} + \tau \norm{g_0}_{L^\infty((0, T), H)}^2) T,
	\end{equation}
	which shows that if $T$ is chosen sufficiently small and proportional to the reciprocal of $\mathcal{K} + \tau \norm{g_0}_{L^\infty((0, T), H)}^2$, we can ensure that $\Phi$ maps $\Gamma^\tau$ into itself. Thus, Banach's fixed-point theorem yields a unique fixed point of $\Phi$, which is a solution to our differential equation. Since the fixed point lies in $\Gamma^\tau$, the energy estimate follows from \cref{eq:GammaEstimate}, and the weak formulation and regularity are direct consequences of \Cref{thm:RelImplExistence}.
\end{proof}


\section{Convergence for \texorpdfstring{$\tau \to 0$}{τ → 0}} \label{sec:Convergence}

In this section we compare solutions $y$ and $y^\tau$ to the unrelaxed model from \cref{eq:GeneralModel}
\begin{equation}\label{eq:ConvUnrel}
	\begin{gathered}
		\delta y + L(a, x) y + \Lambda(a, x, y) y = \sigma(a) \laplace y, \\
		y(t = 0) = y_0, \quad \partial_\nu y(x \in \partial \Omega) = 0\\
		y(t, a = 0, x) = \int_0^\amax \beta(\alpha, x) y(t, \alpha, x) \dd{\alpha}
	\end{gathered}
\end{equation} 
and the relaxed model from \cref{eq:relaxedEquation} and \cref{eq:NonlRelSystem} \begin{equation}\label{eq:ConvRel}
	\begin{gathered}
		(1 + \tau \delta)(\delta y^\tau + L(a, x) y^\tau + \Lambda(a, x, y^\tau) y^\tau) = \sigma(a) \laplace y^\tau,\\
		y^\tau(t = 0) = y_0, \quad \delta y^\tau(t = 0) = y_1, \quad \partial_\nu y^\tau(x \in \partial \Omega) = 0,\\
		y^\tau(t, a = 0, x) = \int_0^\amax \beta_0(\alpha, x) y^\tau(t, \alpha, x) \dd{\alpha} + g_0,\\
		\textcolor{black}{\begin{split}
				(\delta y^\tau)(a = 0) &= \int_0^\amax \beta_L(\alpha, x) y^\tau(t, \alpha, x) + \beta_\nabla(\alpha, x) \cdot \nabla y^\tau(t, \alpha, x) + \beta_1(\alpha, x) (\delta y^\tau)(t, \alpha, x)\\
				&\quad + \qty(\beta_1(\alpha, x) \Lambda(\alpha, x, y^\tau) - \Lambda(0, x, y^\tau) \beta(\alpha, x)) y^\tau(t, \alpha, x) \dd{\alpha} - \Lambda(0, y) g_0(t, x) + g_1(t, x)
		\end{split} }
	\end{gathered}
\end{equation}
We always assume that \Cref{assump:final} holds. From \cite[Thm. 2.15]{AzmiSchlosser} and \Cref{thm:RelMain} we conclude that there exists a time $T$ independent of $\tau$ such that $y \in L^2((0, T), \vilbert) \cap C([0, T], \hilbert)$ and $\delta y \in L^2((0, T), \vilbert')$ as well as $y^\tau \in C([0, T], \vilbert)$ and $\delta y^\tau \in C([0, T], \hilbert)$ for all $\tau \in (0, \tau_{\text{max}})$. The following results, however, require more regularity of $y$.


For the convergence result to hold it is necessary that the relaxed and the unrelaxed equations have matching initial and boundary conditions. Let $q_1 \in \R$, then from \begin{equation}
	y(a = 0) = q_1 y(a = 0) + (1-q_1) y(a = 0) = \int_0^\amax q_1 \beta(\alpha) y(\alpha) \dd{\alpha} + (1-q_1) y(a = 0)
\end{equation}
we infer that the implicit birth conditions for $y$ and $y^\tau$ match if we choose $\beta_0 \coloneqq q_1 \beta$ and $g_0 \coloneqq (1-q_1) y(a = 0)$. In other words, $q_1$ allows to switch between implicit birth conditions for the relaxed equation and explicitly given birth numbers from the unrelaxed model. By applying $(1 + \tau \delta)$ to the first line of \cref{eq:ConvUnrel}, we conclude that $y$ also satisfies \begin{equation} \label{eq:conv1System}
	\begin{aligned}
		\MoveEqLeft (1 + \tau \delta) (\delta y + L y + \Lambda(y) y) = \sigma \laplace y + \tau \delta \sigma \laplace y,\\
		y(t = 0) &= y_0, \quad \delta y(t = 0) = \sigma \laplace y_0 - (L + \Lambda(y_0)) y_0,\\
		y(a = 0, x) &= \int_0^\amax \beta_0(\alpha, x) y(t, \alpha, x) \dd{\alpha} + g_0,\\
				(\delta y)(a = 0) &= \int_0^\amax \beta_L(\alpha, x) y(t, \alpha, x) + \beta_\nabla(\alpha, x) \cdot \nabla y(t, \alpha, x) + \beta_1(\alpha, x) (\delta y)(t, \alpha, x)\\
				&\quad + \qty(\beta_1(\alpha, x) \Lambda(\alpha, x, y) - \Lambda(0, x, y) \beta(\alpha, x)) y(t, \alpha, x) \dd{\alpha} - \Lambda(0, y) g_0(t, x) + \tilde{g}_1(t, x)\\
		\partial_\nu y(x \in \partial \Omega) &= 0
	\end{aligned}
\end{equation}
where we set \begin{align}
	\tilde{g}_1 &\coloneqq (\delta y)(a = 0) - \int_0^\amax \beta_L(\alpha, x) y(t, \alpha, x) + \beta_\nabla(\alpha, x) \cdot \nabla y(t, \alpha, x) + \beta_1(\alpha, x) (\delta y)(t, \alpha, x)\\
	&\quad + \qty(\beta_1(\alpha, x) \Lambda(\alpha, x, y) - \Lambda(0, x, y) \beta(\alpha, x)) y(t, \alpha, x) \dd{\alpha} + \Lambda(0, y) g_0(t, x). \label{eq:g1Def}
\end{align}
The resulting equation has the form of \cref{eq:NonlRelSystem} with \begin{equation}
	\begin{aligned}
		f &= \tau \delta \sigma \laplace y,\\
		y_1 &= \sigma \laplace y_0 - (L + \Lambda(y_0)) y_0,
	\end{aligned} \label{eq:Convy1}
\end{equation}
and $g_1 = \tilde{g}_1.$ This allows a first convergence result:
\begin{Thm} \label{thm:ConvergenceSuboptimal}
	Under \Cref{assump:final}, let $y^\tau$ be a weak solution of \cref{eq:ConvRel} and $y$ a weak solution of \cref{eq:ConvUnrel}. Further we assume that $y \in C([0, T], \vilbert)$ with $\delta y \in C([0, T], \hilbert)$, that \begin{gather}
		\tau \delta \sigma \laplace y \in L^2((0, T), \hilbert), \quad \sigma \laplace y_0 - (L + \Lambda(y_0)) y_0 \in \hilbert, \quad \tilde{g}_1 \in L^2((0, T), H)
	\end{gather}
	where $\tilde{g}_1$ is defined by \cref{eq:g1Def}, and that $\beta_0 = q_1 \beta$ and $g_0 = (1 - q_1) y(a = 0) \in L^\infty((0, T), H) \cap L^2((0, T), V)$ for some $q_1 \in \R$. Then there is a constant $C$ independent of $\tau$ (but possibly depending on $y$ and $\tau_{\text{max}}$) such that \begin{equation}
		\tau \norm{\delta y^\tau - \delta y}_{C([0, T], \hilbert)}^2 + \norm{y^\tau - y}_{C([0, T], \vilbert)}^2 \leq C \tau.
	\end{equation}
	Hence, $y^\tau$ converges in $C([0, T], \vilbert)$ with rate $\sqrt{\tau}$ to $y$ as $\tau \to 0$.
\end{Thm}
\begin{proof}
	Applying \Cref{lem:NonlRelGronwall} to \cref{eq:ConvRel} and \cref{eq:conv1System} directly gives the estimate \begin{align}
		\MoveEqLeft\tau \norm{\delta y^\tau - \delta y}_{C([0, T], \hilbert)}^2 + \norm{y^\tau - y}_{C([0, T], \vilbert)}^2\\*
		&\leqc \tau \norm{y_1 - \sigma \laplace y_0 + (L + \Lambda(y_0)) y_0}_\hilbert^2 + \norm{\tau \delta \sigma \laplace y}_{L^2((0, T), \hilbert)}^2 + \tau \norm{g_1 - \tilde{g}_1}_{L^2((0, T), H)}^2.
	\end{align}
	Thus, choosing $C$ as a multiple of \begin{equation}
		\max\set{\norm{y_1 - \sigma \laplace y_0 + (L + \Lambda(y_0)) y_0}_\hilbert^2, \tau_{\text{max}} \norm{\delta \sigma \laplace y}_{L^2((0, T), \hilbert)}^2, \norm{g_1 - \tilde{g}_1}_{L^2((0, T), H)}^2}
	\end{equation}
	concludes the proof.
\end{proof}

The convergence rate can be improved if the first-order initial and birth conditions for $y$ and $y^\tau$ are compatible too. For this, we need to assume that $\beta$, $\beta_0$ and $\beta_1$ do not depend on space but only on age. Then, \textcolor{black}{we know from \cref{eq:horridCondition} that} $y$ satisfies a first-order implicit birth condition of the form \textcolor{black}{\begin{align}
		(\delta y)(a = 0) &= \int_0^\amax \qty(\sigma(0) \laplace \beta(\alpha, x) + \sigma(0) \beta(\alpha, x) \sigma(\alpha)\inv L(\alpha, x) - L(0, x) \beta(\alpha, x)) y(t, \alpha, x)\\
		&\quad + 2 \sigma(0) \nabla \beta(\alpha, x) \cdot \nabla y(t, \alpha, x) + \sigma(0) \beta(\alpha, x) \sigma(\alpha)\inv (\delta y)(t, \alpha, x)\\
		&\quad + \qty(\sigma(0) \beta(\alpha, x) \sigma(\alpha)\inv \Lambda(\alpha, x, y) - \Lambda(0, x, y) \beta(\alpha, x)) y(t, \alpha, x) \dd{\alpha}.
\end{align}}
This allows for a similar trick as before: Let $q_2 \in \R$, then from \begin{align}
	\MoveEqLeft (\delta y) (a = 0) = q_2 (\delta y) (a = 0) + (1-q_2) (\delta y) (a = 0)\\
	&= \int_0^\amax q_2 \qty(\sigma(0) \laplace \beta(\alpha, x) + \sigma(0) \beta(\alpha, x) \sigma(\alpha)\inv L(\alpha, x) - L(0, x) \beta(\alpha, x)) y(t, \alpha, x)\\
	&\quad + 2 q_2  \sigma(0) \nabla \beta(\alpha, x) \cdot \nabla y(t, \alpha, x) + q_2 \sigma(0) \beta(\alpha, x) \sigma(\alpha)\inv (\delta y)(t, \alpha, x)\\
	&\quad + q_2 \qty(\sigma(0) \beta(\alpha, x) \sigma(\alpha)\inv \Lambda(\alpha, x, y) - \Lambda(0, x, y) \beta(\alpha, x)) y(t, \alpha, x) \dd{\alpha} + (1-q_2) (\delta y) (a = 0).
\end{align}
we conclude that the first-order birth conditions for $y^\tau$ and $y$ match if we choose \begin{equation}
	\begin{aligned}
		\beta_1(\alpha) &\coloneqq q_2 \sigma(0) \beta(\alpha) \sigma(\alpha)\inv,\\
		\beta_L(\alpha) &\coloneqq q_2 \qty(\sigma(0) \laplace \beta(\alpha, x) + \sigma(0) \beta(\alpha, x) \sigma(\alpha)\inv L(\alpha, x) - L(0, x) \beta(\alpha, x)),\\
		\beta_\nabla &\coloneqq 2 q_2  \sigma(0) \nabla \beta(\alpha, x),\\
		g_1 &\coloneqq (1-q_2) (\delta y) (a = 0).
	\end{aligned} \label{eq:KompKond}
\end{equation}
Further, we need to assume that \cref{eq:Convy1} is satisfied. Under these conditions, $y$ solves the system \begin{align}
	\MoveEqLeft (1 + \tau \delta) (\delta y + L y + \Lambda(y) y) = \sigma \laplace y + \tau \delta \sigma \laplace y,\\
	y(t = 0) &= y_0, \quad \delta y(t = 0) = y_1,\\
	y(a = 0, x) &= \int_0^\amax \beta_0(\alpha, x) y(t, \alpha, x) \dd{\alpha} + g_0,\\
		(\delta y)(a = 0) &= \int_0^\amax \beta_L(\alpha, x) y(t, \alpha, x) + \beta_\nabla(\alpha, x) \cdot \nabla y(t, \alpha, x) + \beta_1(\alpha, x) (\delta y)(t, \alpha, x)\\
		&\quad + \qty(\beta_1(\alpha, x) \Lambda(\alpha, x, y) - \Lambda(0, x, y) \beta(\alpha, x)) y(t, \alpha, x) \dd{\alpha} - \Lambda(0, y) g_0(t, x) + g_1(t, x),\\
	\partial_\nu y(x \in \partial \Omega) &= 0.
\end{align}
This formulation allows to prove the following theorem.
\begin{Thm}\label{thm:tauConvergence}
	In the situation of \Cref{thm:ConvergenceSuboptimal}, assume that $\beta$, $\beta_0$ and $\beta_1$ do not depend on $x$. Let $q_2 \in \R$ and assume that the compatibility conditions from \cref{eq:Convy1} and \cref{eq:KompKond} hold. Then we can estimate \begin{equation}
		\tau \norm{\delta y^\tau - \delta y}_{C([0, T], \hilbert)}^2 + \norm{y^\tau - y}_{C([0, T], \vilbert)}^2 \leq C \tau^2
	\end{equation}
	with a constant $C$ independent of $\tau$, but possibly depending on $y$. From this follows that $y^\tau$ converges in $C([0, T], \vilbert)$ to $y$ with rate $\tau$, and $\delta y^\tau$ converges to $\delta y$ in $C([0, T], \hilbert)$ with rate $\sqrt{\tau}$.
\end{Thm}
\begin{proof}
	Invoking \Cref{lem:NonlRelGronwall} shows that \begin{equation}
		\tau \norm{\delta y^\tau - \delta y}_{C([0, T], \hilbert)}^2 + \norm{y^\tau - y}_{C([0, T], \vilbert)}^2 \leqc \tau^2 \norm{\delta \sigma \laplace y}_{L^2((0, T), \hilbert)}^2.
	\end{equation}
	This concludes the proof.
\end{proof}

\section{Numerical simulation} \label{sec:Numerics}
This section addresses some of the questions posed earlier, such as \begin{itemize}
	\item Does the relaxed model actually have finite propagation speed?
	\item Is it worth to consider a relaxed model, or are the solutions \enquote{close enough} to the simpler unrelaxed model that the extra effort is not worth it?
	\item Can we, in some way, visualize the convergence for $\tau$ getting close to zero?
\end{itemize}
Rather than giving universal results on our very general class of models, which is difficult due to the high complexity, we consider the simple SVIR model (which is a simplified version of the model presented in \cite{AbRaSch}) \begin{align}
	\delta S - c V + \left[\mu + \Lambda(I) \right] S &= \sigma_S \Delta S,\\
	\delta V + \left[\mu + c + \phi_1 \Lambda(I) \right] V &= \sigma_V \Delta V,\\
	\delta I - \Lambda(I) (S + \phi_1 V + \phi_2 R) + (\mu + \delta + \gamma) I &= \sigma_I \Delta I,\\
	\delta R - \gamma I + \left[\mu + \phi_2 \Lambda(I) \right] R &= \sigma_R \Delta R,
\end{align} where \begin{equation}
	\Lambda(a, x, I(t, \cdot, \cdot)) = \int_0^\amax \int_\Omega \lambda(a, \alpha, x, \xi) I(t, \alpha, \xi) \dd{\xi} \dd{\alpha}.
\end{equation}
In \cite{AzmiSchlosser} it was shown that this model can be written in the form of \cref{eq:GeneralModel}. The corresponding relaxed version of this model is \begin{align}
	(1 + \tau \delta) (\delta S - c V + \left[\mu + \Lambda(I) \right] S) &= \sigma_S \Delta S,\\
	(1 + \tau \delta) (\delta V + \left[\mu + c + \phi_1 \Lambda(I) \right] V) &= \sigma_V \Delta V,\\
	(1 + \tau \delta) (\delta I - \Lambda(I) (S + \phi_1 V + \phi_2 R) + (\mu + \delta + \gamma) I) &= \sigma_I \Delta I,\\
	(1 + \tau \delta) (\delta R - \gamma I + \left[\mu + \phi_2 \Lambda(I) \right] R) &= \sigma_R \Delta R,
\end{align} We proceed by calculating numerical solutions to both equations and compare the results for varying values of $\tau$ to the solution of the unrelaxed equation where $\tau = 0$.

The space domain $\Omega$ is chosen as the interval $(0, 1)$. The age-dependent functions for the birth and death rate were taken from \cite{AnitaArnautuCapasso}. We choose resp. adapt the model parameters given in \Cref{tab:values} (the same ones as in \cite[Table 1]{AzmiSchlosser}) and the following initial value: A total of 1000 susceptible individuals is distributed uniformly over age and space, and on the boundary where $x = 1$, there are 10 infective individuals, uniformly distributed over all ages. This allows us to track the progression of the epidemic through space and to compare the differing propagation speed. We test various values for $\tau$, namely all $\tau = 10^{j}$ for $j \in \set{-8, \ldots, 2}$. Also, we choose $y_1 = 0$ and $\beta_1 = \beta_0 = \beta$, which can be interpreted as follows: At $t = 0$, the same number of individuals move to the left and to the right, and newborn individuals move in the same direction as their parents. A simple alternative choice would be $\beta_1 = 0$, which  means that the direction of newborns is determined at random. Thus, the compatibility conditions of \Cref{thm:ConvergenceSuboptimal} are satisfied with $q_1 = 1$.


\begin{table}[h]
	\centering
	\begin{tabular}{cll}\toprule
		Parameter & Value & Source\\\midrule
		$T$ & 5 & Assumed\\
		$\amax$ & 1 & Assumed\\
		$\alpha$ & 500 & Assumed\\
		$c$ & 0.18564 & \cite[Table 1]{AbRaSch}\\
		$\mu$ & $e^{-a} \cdot a^5$ & Adapted from \cite[p. 155]{AnitaArnautuCapasso}, without the pole at $\amax$\\
		$\phi_1$ & 0.0052 & \cite[Table 1]{AbRaSch}\\
		$\phi_2$ & 0.00062 & \cite[Table 1]{AbRaSch}\\
		$\delta$ & 0.0018 & \cite[Table 1]{AbRaSch}\\ 
		$\gamma$ & 0.278574 & $\theta/2$ from \cite[Table 1]{AbRaSch}\\
		$\lambda$ & $(0.1 - \abs{x - \xi})^+$ & Assumed\\
		$\beta$ & $\frac{6.78}{\amax} a^2 (\amax - a) (1 + \sin(\pi \frac{a}{\amax}))$ & Adapted from \cite[p. 155]{AnitaArnautuCapasso}\\
		$\sigma_S$ & $0.1 e^{-0.1 a}$ & Assumed\\
		$\sigma_V$ & $0.1 e^{-0.1 a}$ & Assumed\\
		$\sigma_I$ & $0.05 e^{-0.1 a}$ & Assumed\\
		$\sigma_R$ & $0.1 e^{-0.1 a}$ & Assumed\\
		\bottomrule
	\end{tabular}
	\caption{Parameter Setting}
	\label{tab:values}
\end{table}

\begin{figure}
	\centering
	\begin{subfigure}{.8\textwidth}
		\includegraphics[width = \linewidth]{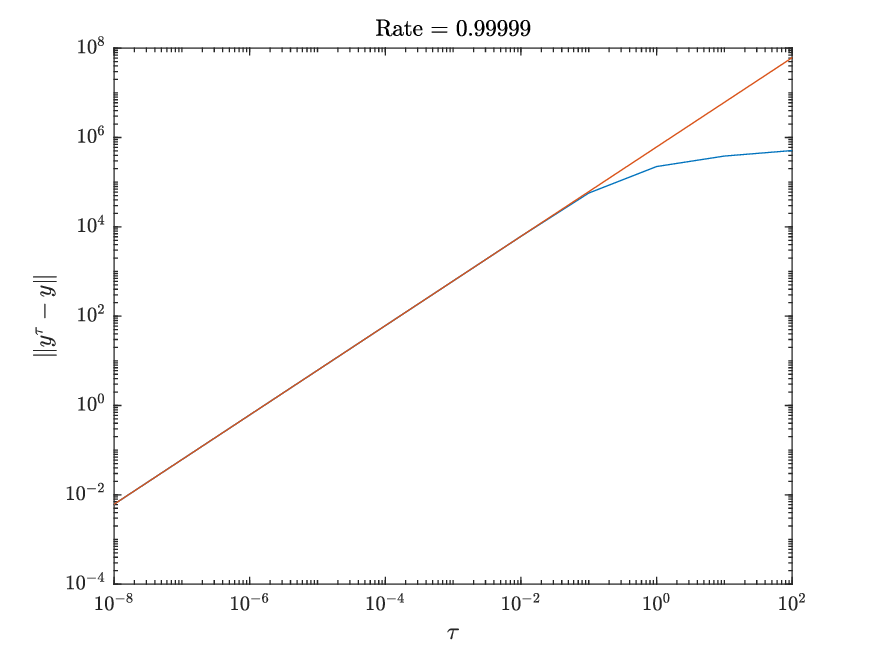}
		\caption{Convergence rate}
		\label{fig:CompRate}
	\end{subfigure}
	\begin{subfigure}{.8\textwidth}
		\includegraphics[width = \linewidth]{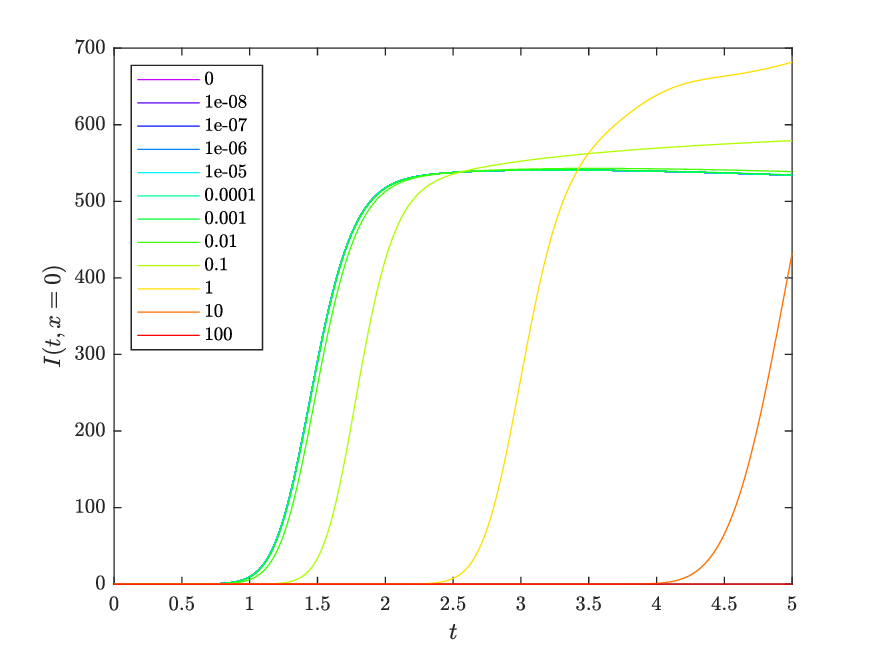}
		\caption{Values of $I$ at $x = 0$}
		\label{fig:CompValuesI}
	\end{subfigure}
\end{figure}
\begin{figure} \ContinuedFloat
	\centering
	\begin{subfigure}{\textwidth}
		\includegraphics[width = \linewidth]{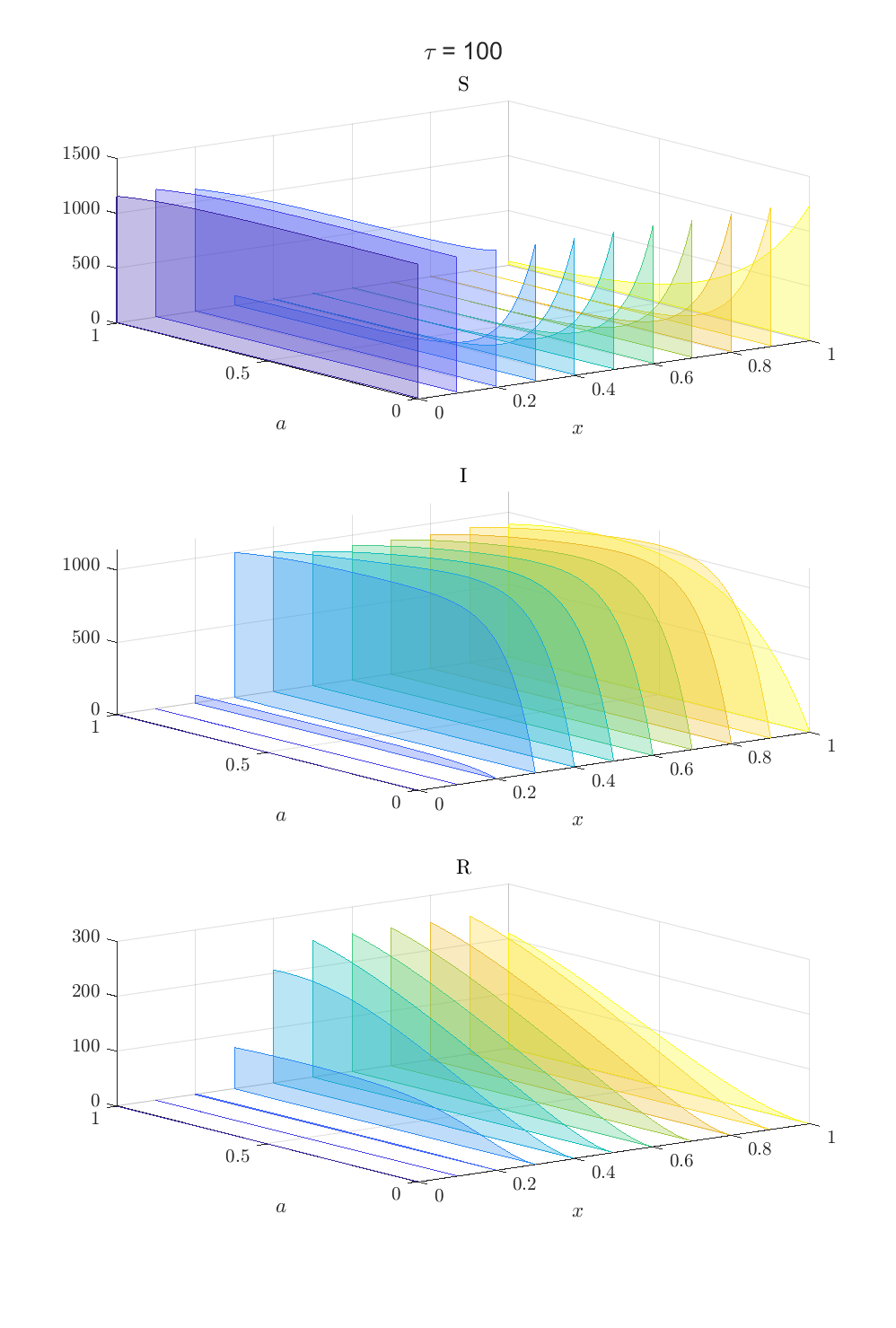}
		\caption{Final state of the relaxed equation}
		\label{fig:CompRelaxedFinal}
	\end{subfigure}
\end{figure}
\begin{figure} \ContinuedFloat
	\centering
	\begin{subfigure}{\textwidth}
		\includegraphics[width = \linewidth]{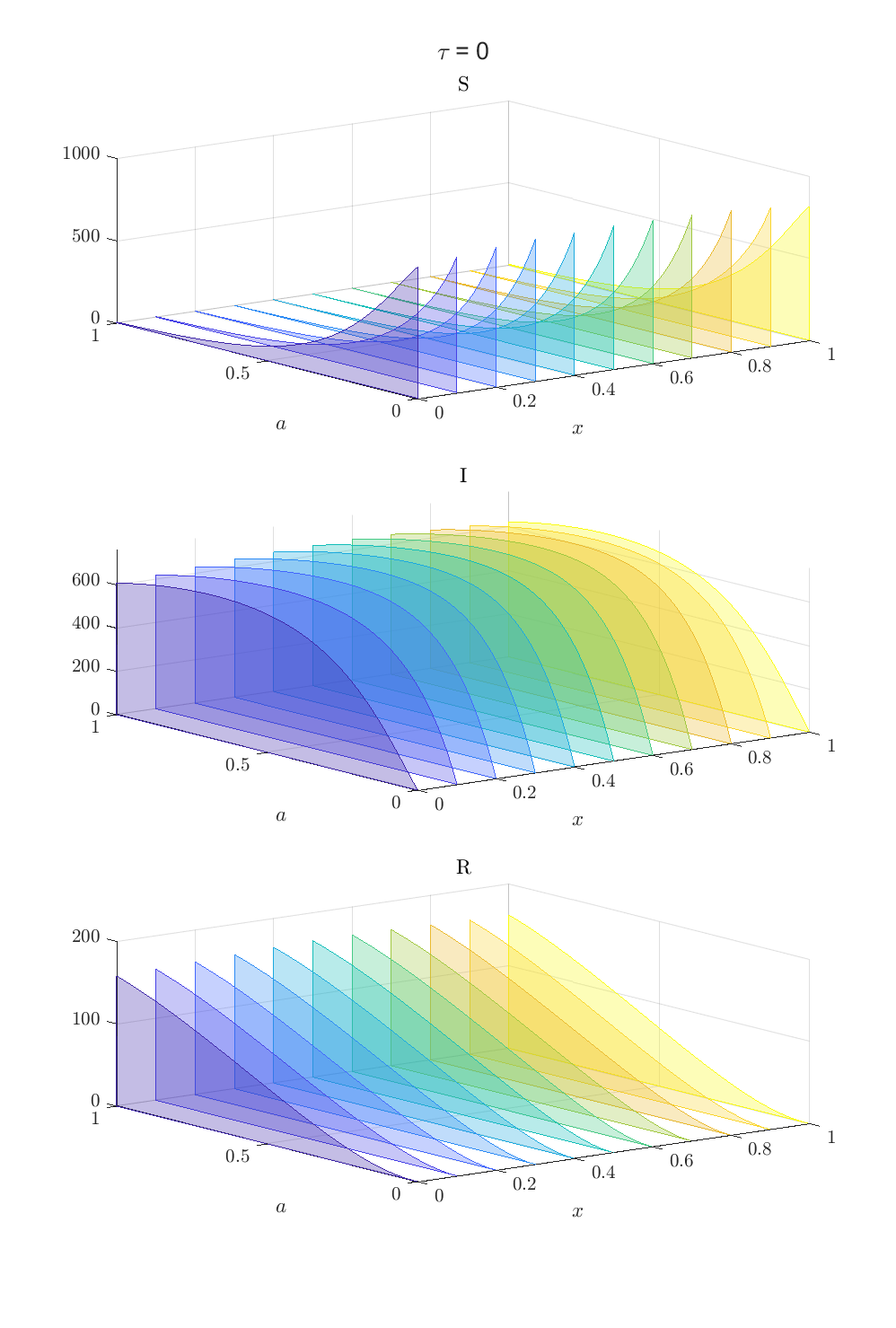}
		\caption{Final state of the unrelaxed equation}
		\label{fig:CompUnrelaxedFinal}
	\end{subfigure}
	\caption{Comparison of different values of $\tau$}
	\label{fig:Comp}
\end{figure}

The results can be found in \Cref{fig:Comp}. In \Cref{fig:CompRate}, we numerically calculated the rate of convergence of the unrelaxed model towards the relaxed one. The figure shows a graph of $\norm{y^\tau - y}_{L^\infty}$ over $\tau$, and a polynomial fit of the data. The resulting slope, and hence the convergence rate, turns out to be 1, which is even better than the value of 0.5 we would expect from \Cref{thm:ConvergenceSuboptimal}, and suggests that the results from \Cref{thm:tauConvergence} hold in higher generality. \Cref{fig:CompValuesI} shows the total number of infective individuals at the $x = 0$ boundary, directly opposite to where all infectives have been in the beginning of the simulation. The lines corresponding to smaller values of $\tau$ all overlap, but the lines for $\tau = 0.1$, $\tau = 1$ and $\tau = 10$ are clearly distinct and show that the larger $\tau$ is, the slower the infection moves through the domain. The line corresponding to $\tau = 100$ is just flat at the bottom, the population of infective individuals has not yet fully traversed the interval. This can also be seen in \Cref{fig:CompRelaxedFinal}, which shows the state for $\tau = 100$ at the final time of the simulation. Here, the progress of the infection can directly be observed, it seems to have just arrived at the $x = 0.2$ mark. Compare this result to \Cref{fig:CompUnrelaxedFinal}, which shows the final state of the solution for $\tau = 0$, where no real spatial difference can be seen. To summarize, our results suggest that the relaxed model does in fact show finite propagation speed and hence is very suitable to model the spread of an epidemic over a region, but choosing small values for $\tau$ yields results that do not differ very much from an unrelaxed model.

The graphs, especially \Cref{fig:CompValuesI}, also show that the peak number of infectious individuals is higher for large values of tau. However, note that this is not only true for the infective population. In fact, for the values of $\tau$ larger than 0.01, the total population increases, while for smaller values it behaves like the unrelaxed model, that is, it decreases slightly. This can also be seen from \Cref{fig:CompRelaxedFinal}, where, for example, the number of susceptible numbers visibly exceeds the initial value of 1000 individuals, while in \Cref{fig:CompUnrelaxedFinal} it remains below that number. This suggests that either large values of $\tau$ yield a model that deviates too much from the unrelaxed equation to be useful, or that the birth rates need to be adjusted for these values.


\vspace{1em}
\textbf{Conflict of interest:} The author has not disclosed any competing interests.

\bibliographystyle{abbrvurl}
\setbibliographyfont{lastname}{\textsc}
\bibliography{Bibliography}

\begin{thebibliography}{10}

\bibitem{AbRaSch}
H.~Abboubakar, R.~Racke, and N.~Schlosser.
\newblock Ode and pde models for covid-19, with reinfection and vaccination
  process for cameroon and germany, 2025.
\newblock URL: \url{https://arxiv.org/abs/2504.21613}, \href
  {https://arxiv.org/abs/2504.21613} {\path{arXiv:2504.21613}}.

\bibitem{AdamsFournier}
R.~A. Adams and J.~J.~F. Fournier.
\newblock {\em Sobolev spaces}, volume 140 of {\em Pure and Applied Mathematics
  (Amsterdam)}.
\newblock Elsevier/Academic Press, Amsterdam, second edition, 2003.

\bibitem{AllharbiPetrovskii}
W.~Alharbi and S.~Petrovskii.
\newblock Critical domain problem for the reaction–telegraph equation model
  of population dynamics.
\newblock {\em Mathematics}, 6(4), 2018.
\newblock URL: \url{https://www.mdpi.com/2227-7390/6/4/59}, \href
  {https://doi.org/10.3390/math6040059} {\path{doi:10.3390/math6040059}}.

\bibitem{AnitaArnautuCapasso}
S.~Ani\c{t}a, V.~Arn\u{a}utu, and V.~Capasso.
\newblock {\em An introduction to optimal control problems in life sciences and
  economics}.
\newblock Modeling and Simulation in Science, Engineering and Technology.
  Birkh\"auser/Springer, New York, 2011.
\newblock From mathematical models to numerical simulation with
  MATLAB$^\circledR$.
\newblock \href {https://doi.org/10.1007/978-0-8176-8098-5}
  {\path{doi:10.1007/978-0-8176-8098-5}}.

\bibitem{AzmiSchlosser}
B.~Azmi and N.~Schlosser.
\newblock Optimal control of general nonlocal epidemic models with age and
  space structure, 2025.
\newblock URL: \url{https://arxiv.org/abs/2503.01466}, \href
  {https://arxiv.org/abs/2503.01466} {\path{arXiv:2503.01466}}.

\bibitem{BDKW}
S.~Bentout, S.~Djilali, T.~Kuniya, and J.~Wang.
\newblock Mathematical analysis of a vaccination epidemic model with nonlocal
  diffusion.
\newblock {\em Math. Methods Appl. Sci.}, 46(9):10970--10994, 2023.
\newblock \href {https://doi.org/10.1002/mma.9162}
  {\path{doi:10.1002/mma.9162}}.

\bibitem{BertagliaPareschi}
G.~Bertaglia and L.~Pareschi.
\newblock Hyperbolic models for the spread of epidemics on networks: kinetic
  description and numerical methods.
\newblock {\em ESAIM: M2AN}, 55(2):381--407, 2021.
\newblock \href {https://doi.org/10.1051/m2an/2020082}
  {\path{doi:10.1051/m2an/2020082}}.

\bibitem{Bongarti}
M.~Bongarti, C.~Parkinson, and W.~Wang.
\newblock Optimal control of a reaction-diffusion epidemic model with
  noncompliance, 2024.
\newblock URL: \url{https://arxiv.org/abs/2407.17298}, \href
  {https://arxiv.org/abs/2407.17298} {\path{arXiv:2407.17298}}.

\bibitem{BDJ}
F.~Brauer, P.~van~den Driessche, and J.~Wu, editors.
\newblock {\em Mathematical epidemiology}, volume 1945 of {\em Lecture Notes in
  Mathematics}.
\newblock Springer-Verlag, Berlin, 2008.
\newblock Mathematical Biosciences Subseries.
\newblock \href {https://doi.org/10.1007/978-3-540-78911-6}
  {\path{doi:10.1007/978-3-540-78911-6}}.

\bibitem{Cattaneo}
C.~Cattaneo.
\newblock Sulla conduzione del calore.
\newblock {\em Atti Sem. Mat. Fis. Univ. Modena}, 3:83--101, 1948.

\bibitem{Colombo}
R.~M. Colombo, M.~Garavello, F.~Marcellini, and E.~Rossi.
\newblock An age and space structured sir model describing the covid-19
  pandemic.
\newblock {\em Journal of Mathematics in Industry}, 10(1):22, Aug 2020.
\newblock \href {https://doi.org/10.1186/s13362-020-00090-4}
  {\path{doi:10.1186/s13362-020-00090-4}}.

\bibitem{DautrayLions}
R.~Dautray and J.-L. Lions.
\newblock {\em Mathematical analysis and numerical methods for science and
  technology. {V}ol. 5}.
\newblock Springer-Verlag, Berlin, 1992.
\newblock Evolution problems. I, With the collaboration of Michel Artola,
  Michel Cessenat and H\'{e}l\`ene Lanchon, Translated from the French by Alan
  Craig.
\newblock \href {https://doi.org/10.1007/978-3-642-58090-1}
  {\path{doi:10.1007/978-3-642-58090-1}}.

\bibitem{DreherQRacke}
M.~Dreher, R.~Quintanilla, and R.~Racke.
\newblock Ill-posed problems in thermomechanics.
\newblock {\em Applied Mathematics Letters}, 22(9):1374--1379, 2009.
\newblock URL:
  \url{https://www.sciencedirect.com/science/article/pii/S0893965909001293},
  \href {https://doi.org/10.1016/j.aml.2009.03.010}
  {\path{doi:10.1016/j.aml.2009.03.010}}.

\bibitem{EngelNagel}
K.-J. Engel and R.~Nagel.
\newblock {\em One-parameter semigroups for linear evolution equations}, volume
  194 of {\em Graduate Texts in Mathematics}.
\newblock Springer-Verlag, New York, 2000.
\newblock With contributions by S. Brendle, M. Campiti, T. Hahn, G. Metafune,
  G. Nickel, D. Pallara, C. Perazzoli, A. Rhandi, S. Romanelli and R.
  Schnaubelt.

\bibitem{Fragnelli}
G.~Fragnelli.
\newblock An age-dependent population equation with diffusion and delayed birth
  process.
\newblock {\em International Journal of Mathematics and Mathematical Sciences},
  2005(20):409340, 2005.
\newblock \href
  {https://arxiv.org/abs/https://onlinelibrary.wiley.com/doi/pdf/10.1155/IJMMS.2005.3273}
  {\path{arXiv:https://onlinelibrary.wiley.com/doi/pdf/10.1155/IJMMS.2005.3273}},
  \href {https://doi.org/10.1155/IJMMS.2005.3273}
  {\path{doi:10.1155/IJMMS.2005.3273}}.

\bibitem{Hadeler2}
K.~Hadeler.
\newblock Reaction telegraph equations and random walk systems.
\newblock {\em Stochastic and spatial structures of dynamical systems},
  45:133--161, 1996.

\bibitem{Hadeler}
K.~Hadeler.
\newblock Reaction transport equations in biological modeling.
\newblock {\em Mathematical and Computer Modelling}, 31(4):75--81, 2000.
\newblock Proceedings of the Conference on Dynamical Systems in Biology and
  Medicine.
\newblock URL:
  \url{https://www.sciencedirect.com/science/article/pii/S0895717700000248},
  \href {https://doi.org/10.1016/S0895-7177(00)00024-8}
  {\path{doi:10.1016/S0895-7177(00)00024-8}}.

\bibitem{HPUU}
M.~Hinze, R.~Pinnau, M.~Ulbrich, and S.~Ulbrich.
\newblock {\em Optimization with {PDE} constraints}, volume~23 of {\em
  Mathematical Modelling: Theory and Applications}.
\newblock Springer, New York, 2009.

\bibitem{Holmes}
E.~E. Holmes.
\newblock Are diffusion models too simple? a comparison with telegraph models
  of invasion.
\newblock {\em The American Naturalist}, 142(5):779--795, 1993.
\newblock PMID: 19425956.
\newblock \href {https://arxiv.org/abs/https://doi.org/10.1086/285572}
  {\path{arXiv:https://doi.org/10.1086/285572}}, \href
  {https://doi.org/10.1086/285572} {\path{doi:10.1086/285572}}.

\bibitem{HuRacke}
Y.~Hu and R.~Racke.
\newblock Compressible navier–stokes equations with hyperbolic heat
  conduction.
\newblock {\em Journal of Hyperbolic Differential Equations}, 13(02):233--247,
  2016.
\newblock \href
  {https://arxiv.org/abs/https://doi.org/10.1142/S0219891616500077}
  {\path{arXiv:https://doi.org/10.1142/S0219891616500077}}, \href
  {https://doi.org/10.1142/S0219891616500077}
  {\path{doi:10.1142/S0219891616500077}}.

\bibitem{Kac}
M.~Kac.
\newblock A stochastic model related to the telegrapher's equation.
\newblock {\em The Rocky Mountain Journal of Mathematics}, 4(3):497--509, 1974.
\newblock URL: \url{http://www.jstor.org/stable/44236399}.

\bibitem{KangRuan}
H.~Kang and S.~Ruan.
\newblock Mathematical analysis on an age-structured {SIS} epidemic model with
  nonlocal diffusion.
\newblock {\em J. Math. Biol.}, 83(1):Paper No. 5, 30, 2021.
\newblock \href {https://doi.org/10.1007/s00285-021-01634-x}
  {\path{doi:10.1007/s00285-021-01634-x}}.

\bibitem{KermackMcKendrick}
W.~Kermack and A.~McKendrick.
\newblock Contributions to the mathematical theory of epidemics—i.
\newblock {\em Bltn Mathcal Biology 53, 33–55}, 1991.
\newblock \href {https://doi.org/10.1007/BF02464423}
  {\path{doi:10.1007/BF02464423}}.

\bibitem{Kuniya}
T.~Kuniya and R.~Oizumi.
\newblock Existence result for an age-structured sis epidemic model with
  spatial diffusion.
\newblock {\em Nonlinear Analysis: Real World Applications}, 23:196--208, 2015.
\newblock URL:
  \url{https://www.sciencedirect.com/science/article/pii/S1468121814001448},
  \href {https://doi.org/10.1016/j.nonrwa.2014.10.006}
  {\path{doi:10.1016/j.nonrwa.2014.10.006}}.

\bibitem{LMPS}
C.~Lattanzio, C.~Mascia, R.~G. Plaza, and C.~Simeoni.
\newblock Analysis and numerics of the propagation speed for hyperbolic
  reaction-diffusion models, 2022.
\newblock URL: \url{https://arxiv.org/abs/2206.09714}, \href
  {https://arxiv.org/abs/2206.09714} {\path{arXiv:2206.09714}}.

\bibitem{LiYaMa}
X.-Z. Li, J.~Yang, and M.~Martcheva.
\newblock {\em Age structured epidemic modeling}, volume~52 of {\em
  Interdisciplinary Applied Mathematics}.
\newblock Springer, Cham, [2020] \copyright 2020.
\newblock \href {https://doi.org/10.1007/978-3-030-42496-1}
  {\path{doi:10.1007/978-3-030-42496-1}}.

\bibitem{Maxwell}
J.~C. Maxwell.
\newblock Iv. on the dynamical theory of gases.
\newblock {\em Philosophical Transactions of the Royal Society of London},
  157:49--88, 1867.
\newblock URL:
  \url{https://royalsocietypublishing.org/doi/abs/10.1098/rstl.1867.0004},
  \href
  {https://arxiv.org/abs/https://royalsocietypublishing.org/doi/pdf/10.1098/rstl.1867.0004}
  {\path{arXiv:https://royalsocietypublishing.org/doi/pdf/10.1098/rstl.1867.0004}},
  \href {https://doi.org/10.1098/rstl.1867.0004}
  {\path{doi:10.1098/rstl.1867.0004}}.

\bibitem{MelnikovaFilinkov}
I.~V. Mel'nikova and A.~I. Filinkov.
\newblock A connection between the well-posedness of the cauchy problem for an
  equation and for a system in a banach space.
\newblock In {\em Doklady Akademii Nauk}, volume 300, pages 280--284. Russian
  Academy of Sciences, 1988.

\bibitem{Perthame}
B.~Perthame.
\newblock {\em Parabolic Equations in Biology: Growth, reaction, movement and
  diffusion}.
\newblock Lecture Notes on Mathematical Modelling in the Life Sciences.
  Springer International Publishing, 2015.
\newblock URL: \url{https://books.google.de/books?id=0pOKCgAAQBAJ}.

\bibitem{Pogorui}
A.~A. Pogorui and R.~M. Rodríguez-Dagnino.
\newblock Goldstein-kac telegraph equations and random flights in higher
  dimensions.
\newblock {\em Applied Mathematics and Computation}, 361:617--629, 2019.
\newblock URL:
  \url{https://www.sciencedirect.com/science/article/pii/S0096300319304473},
  \href {https://doi.org/10.1016/j.amc.2019.05.045}
  {\path{doi:10.1016/j.amc.2019.05.045}}.

\bibitem{Pruess}
J.~Pr{\"u}ss.
\newblock {\em Evolutionary integral equations and applications}, volume~87 of
  {\em Monographs in Mathematics}.
\newblock Birkh\"auser Verlag, Basel, 1993.
\newblock \href {https://doi.org/10.1007/978-3-0348-8570-6}
  {\path{doi:10.1007/978-3-0348-8570-6}}.

\bibitem{RackeDelay}
R.~Racke.
\newblock Instability of coupled systems with delay.
\newblock {\em Communications on Pure and Applied Analysis}, 11(5):1753--1773,
  2012.
\newblock URL:
  \url{https://www.aimsciences.org/article/id/296b97e0-4e28-4de7-9e42-eadd75223021},
  \href {https://doi.org/10.3934/cpaa.2012.11.1753}
  {\path{doi:10.3934/cpaa.2012.11.1753}}.

\bibitem{RackeNonlinear}
R.~Racke.
\newblock {\em Lectures on Nonlinear Evolution Equations: Initial Value
  Problems}.
\newblock Springer International Publishing, 2015.
\newblock \href {https://doi.org/10.1007/978-3-319-21873-1}
  {\path{doi:10.1007/978-3-319-21873-1}}.

\bibitem{SchlosserMaster}
N.~Schlosser.
\newblock {SIR}-{E}pidemiemodelle mit {A}lters- und {O}rtsstruktur.
\newblock Master's thesis, Universität Konstanz, 2021.

\bibitem{SchlosserThesis}
N.~Schlosser.
\newblock {\em Relaxation and Optimal Control of Age- and Space-structured
  Epidemic Models}.
\newblock PhD thesis, Universität Konstanz, Konstanz, 2025.

\bibitem{Smoller}
J.~Smoller.
\newblock {\em Shock waves and reaction-diffusion equations}, volume 258 of
  {\em Grundlehren der Mathematischen Wissenschaften}.
\newblock Springer-Verlag, New York-Berlin, 1983.

\bibitem{TillesPetrovskii}
P.~F.~C. Tilles and S.~V. Petrovskii.
\newblock On the consistency of the reaction-telegraph process within finite
  domains.
\newblock {\em Journal of Statistical Physics}, 177(4):569--587, Nov 2019.
\newblock \href {https://doi.org/10.1007/s10955-019-02379-0}
  {\path{doi:10.1007/s10955-019-02379-0}}.

\bibitem{Walker}
C.~Walker.
\newblock Some remarks on the asymptotic behavior of the semigroup associated
  with age-structured diffusive populations.
\newblock {\em Monatsh. Math.}, 170(3-4):481--501, 2013.
\newblock \href {https://doi.org/10.1007/s00605-012-0428-3}
  {\path{doi:10.1007/s00605-012-0428-3}}.

\bibitem{Walker2}
C.~Walker.
\newblock Well-posedness and stability analysis of an epidemic model with
  infection age and spatial diffusion.
\newblock {\em J. Math. Biol.}, 87(3):Paper No. 52, 46, 2023.
\newblock \href {https://doi.org/10.1007/s00285-023-01980-y}
  {\path{doi:10.1007/s00285-023-01980-y}}.

\bibitem{WangZhangKuniya}
J.~Wang, R.~Zhang, and T.~Kuniya.
\newblock A reaction-diffusion susceptible-vaccinated-infected-recovered model
  in a spatially heterogeneous environment with {D}irichlet boundary condition.
\newblock {\em Math. Comput. Simulation}, 190:848--865, 2021.
\newblock \href {https://doi.org/10.1016/j.matcom.2021.06.020}
  {\path{doi:10.1016/j.matcom.2021.06.020}}.

\bibitem{Webb}
G.~F. Webb.
\newblock {\em Population Models Structured by Age, Size, and Spatial
  Position}, pages 1--49.
\newblock Springer Berlin Heidelberg, Berlin, Heidelberg, 2008.
\newblock \href {https://doi.org/10.1007/978-3-540-78273-5_1}
  {\path{doi:10.1007/978-3-540-78273-5_1}}.

\bibitem{Zheng}
Q.~Zheng.
\newblock Strongly continuous m, n-families of bounded operators.
\newblock {\em Integral Equations and Operator Theory}, 19(1):105--119, Mar
  1994.
\newblock \href {https://doi.org/10.1007/BF01202292}
  {\path{doi:10.1007/BF01202292}}.

\end{thebibliography}
\end{document}